\documentclass{article}

    \usepackage[utf8]{inputenc}
    \usepackage{verbatim}
    \usepackage{amsfonts}
    \usepackage{amsmath}
    \usepackage{amsthm}
    \usepackage{stmaryrd} 
    
    \usepackage[style=alphabetic,maxbibnames=99]{biblatex}
        \DeclareFieldFormat{title}{\mkbibquote{#1}}
    
    \usepackage[hidelinks]{hyperref}
    \usepackage[noabbrev,nameinlink,capitalise]{cleveref}
        \crefname{subsection}{Subsection}{Subsections}
        \crefname{subsection}{Subsection}{Subsections}
    
    \usepackage{subcaption}
    \usepackage{color}
    \usepackage[dvipsnames]{xcolor}
    \usepackage{tikz}
        \usetikzlibrary{arrows,arrows.meta,shapes}
        \tikzset{every picture/.style=thick}
        \tikzset{vertex/.style = {ellipse,draw,inner sep=1pt, outer sep=1pt}}
        \tikzset{edge/.style = {-Stealth, >={latex[scale=1.5]}}}
        \tikzset{node/.style = {circle,draw,inner sep=1.5pt,outer sep=1.5pt}}
        \tikzset{line/.style = {->, >={latex[scale=1.5]}}}
        \tikzset{point/.style = {circle,draw,fill,inner sep=2pt}}
        \tikzset{strand/.style = {->, >={stealth[scale=1.5]}}}

    \theoremstyle{plain}
        \newtheorem{theorem}{Theorem}[section]

        \newtheorem{proposition}[theorem]{Proposition}
    \theoremstyle{definition}
        \newtheorem{definition}[theorem]{Definition}
        \newtheorem{example}[theorem]{Example}
        \newtheorem{assumption}[theorem]{Assumption}
        \newtheorem{algorithm}[theorem]{Algorithm}
    \theoremstyle{remark}
        \newtheorem{remark}[theorem]{Remark}

    \title{Topological Full Groups of Irreducible Edge Shifts have Solvable Conjugacy Problem}
    \author{Matteo Tarocchi}
    \date{}

    \newcommand{\init}[0]{\mathsf{i}}
    \newcommand{\term}[0]{\mathsf{t}}
    \newcommand{\col}[0]{\mathrm{c}}
    \newcommand{\F}[0]{\mathcal{F}}
    \newcommand{\G}[0]{\mathcal{G}}
    \newcommand{\Shift}[0]{\mathcal{S}}
    \newcommand{\Lang}[0]{\mathbb{L}}
    \newcommand{\Loops}[0]{\mathcal{L}}
    \newcommand{\Groupoid}[0]{\mathfrak{G}}
    \newcommand{\C}[0]{\mathbf{C}}

\begin{document}

\maketitle

\section{Introduction}

In this paper we solve the conjugacy problem for the topological full groups of irreducible edge shifts.
These groups were first studied in \cite{Matui}, where it is proved (among other facts) that they are of type $F_\infty$ (in particular, finitely presented) and that their commutator subgroups are simple and finitely generated.

While these groups were introduced in the setting of étale groupoids, they were later recontextualized as groups of almost automorphisms of trees \cite{Lederle}.
We will stick closer to this second approach, although our definitions are going to differ slightly from those in \cite{Lederle} in order to more easily define strand diagrams.
Moreover, our methods also work without assuming that the edge shift is irreducible, in which case the groups are not of almost automorphisms of forests, but only of homeomorphisms of the boundary.
Ultimately, we solve the conjugacy problem in all groups of \textit{piecewise-canonical homeomorphisms of edge shifts}, which is a larger family of groups than that introduced in \cite{Matui}.

It is the author's opinion that groups of piecewise-canonical homeomorphisms of edge shifts are interesting even when the edge shift is not irreducible, as they include for example the Houghton groups \cite{Houghton} and the Thompson-like group $QV$ \cite{QV,QV1}.
Note that the conjugacy problem in the Houghton groups and in $QV$ was previously solved in \cite{HoughtonConjugacy} and \cite{RearrConj}.

We employ the tool of \textit{strand diagrams}, which was introduced in \cite{BM14} for Thompson's groups $F$, $T$ and $V$ and also used in \cite{Aroca2018TheCP,,RearrConj,,GofferLederle}.
The version of this technology that we will use was developed in \cite{RearrConj}, but additional ideas (mostly those in \cref{sec:loops}) were gathered by the author while writing the upcoming paper \cite{ConjNek}, which is still work-in-progress.
Note that however \cite{RearrConj} alone does not solve the conjugacy problem in the titular groups.
This paper will recap what is needed from \cite{RearrConj}, simplifying it for this setting and generalizing it to allow for isolated points, and will apply a new strategy to fill the gaps left from \cite{RearrConj}, which is about solving for type 3 reductions.

In \cite{RearrConj}, the condition of \textit{reduction-confluence} was assumed and employed to solve for type 3 reductions.
This is a condition of confluence on a graph rewriting system.
For edge shifts, we will see that there is no need to assume such a condition:
type 3 reductions can instead always be dealt with by solving the word problem in some finitely presented commutative semigroups.

\subsection{Basic definitions, notations and conventions}
\label{sub:notation:convention}

We compose on the right, i.e., $fg$ means that $f$ is applied first and $g$ second.

The natural numbers $\mathbb{N}$ include the element $0$.

A \textbf{multiset} with underlying set $S$ is map $\lambda \colon S \to \mathbb{N}$.
Informally, it is of a collection of elements of $S$, each taken with multiplicity determined by $\lambda$.

A \textbf{graph} is a tuple $(V\Gamma,E\Gamma,\init,\term)$, where $V$ is a non-empty set (of \textit{vertices}), $E$ is a set (of \textit{edges}), $\init$ and $\term$ are maps $E \to V$ ($\init(e)$ and $\term(e)$ are the \textit{initial} and \textit{terminal} vertices of $e$, respectively).
If a graph is vertex- or edge-colored, $\col(x)$ will denote the color of the vertex or edge $x$.
The \textbf{in-} and \textbf{out-degrees} of a vertex $v$ are the cardinalities of the sets $\term(v)^{-1}$ and $\init^{-1}(e)$, respectively.

We will work with three distinct families of graphs:
a graph $\Gamma$ that defines an edge shift, a forest $\F(\Gamma|Y)$ of paths on $\Gamma$, and (closed) strand diagrams.
To minimize confusion, we use different names for their vertices and edges:
\begin{itemize}
    \item graphs $\Gamma$ have \textit{vertices} and \textit{edges} (the standard terminology);
    \item the forests of paths $\F(\Gamma|Y)$ have \textit{nodes} and \textit{lines};
    \item strand diagrams have \textit{points} and \textit{strands}.
\end{itemize}
For forests, a node is a \textit{root} if it has in-degree $0$, a \textit{leaf} if it has out-degree $0$, an \textit{internal node} otherwise.
Rooted forests have edges oriented away from roots.


\section{Piecewise-canonical homeomorphisms}

In this section we define the titular groups and the spaces on which they act.

\subsection{Edge shifts}

Recall from \cref{sub:notation:convention} that we call \textit{nodes} the vertices of a forest and \textit{lines} its edges, to avoid confusion with the graphs $\Gamma$.

\begin{definition}
\label{def:edge:shift}
Let $\Gamma = (V\Gamma,E\Gamma,\init,\term)$ be a finite graph and $Y$ be a multiset with underlying set $V\Gamma$.
The \textbf{(multi-initial) edge shift} $\Shift(\Gamma|Y)$ is the set of all infinite paths on $\Gamma$ starting from an element of $Y$ (multiple copies of the same vertex are counted with their multiplicity).
The \textbf{language} of the edge shift, denoted by $\Lang(\Gamma|Y)$, is the set of all finite paths on the graph $\Gamma$ starting from an element of $Y$, including distinct empty paths for all elements of $Y$.
\end{definition}

Each element of $\Shift(\Gamma|Y)$ is identified with an element of $Y$ followed by some infinite sequence of edges of $\Gamma$.
In the same fashion, each element of $\Lang(\Gamma|Y)$ is identified with an element of $Y$ followed by some finite sequence of edges.

Given $w \in \Lang(\Gamma|Y)$, its \textbf{cylinder} $C_w$ is the set of those elements of $\Shift(\Gamma|Y)$ that have $w$ as a prefix.
If $w = y e_1 \dots e_k$, the color of the cylinder $C_w$ is defined as the color of $e_k$ (or is $y$ when $k=0$).
We equip each edge shift $\Shift(\Gamma|Y)$ with the topology generated by the set of all cylinders, which makes it a Stone space.

What is commonly referred to as \textit{the edge shift on $\Gamma$} corresponds to $\Shift(\Gamma|V\Gamma)$.
For more about edge shifts, refer for example to \cite{LindMarcus}.

\subsection{Assumptions on the graph \texorpdfstring{$\Gamma$}{Gamma}}

From here onward, we will always work under the following assumptions.

\begin{assumption}[No dead ends]
\label{ass:no:out:degree:0}
No vertex of $\Gamma$ has out-degree $0$.
\end{assumption}

This assumption does not limit our choice of shift spaces nor groups.
Indeed, if $v$ has out-degree $0$ then any cylinder $C_{e_1 \dots e_k}$ with $\col(e_k)=v$ is empty, so $v$ can be removed from $\Gamma$ without modifying the edge shift.
Under this assumption, $\Lang(\Gamma|Y)$ is the set of all prefixes of the elements of $\Shift(\Gamma|Y)$.

\begin{assumption}[No redundant edges]
\label{ass:no:out:degree:1}
If $v$ is a vertex of $\Gamma$ of out-degree $1$, then $\init^{-1}(v)$ consists solely of a loop (i.e., an edge $e$ with $\init(e)=\term(e)$).
\end{assumption}

This assumption too does not limit our choice.
If $\Gamma$ has a vertex $v$ such that $\init^{-1}(v)=\{e\}$ and $\term(e) \neq v$, one can contract the edge $e$, identifying the vertices $v=\init(e)$ and $\term(e)$.
In terms of the edge shift, this modification simply recodes each word $pe$ into $p$, for all paths $p$ in $\Gamma$ terminating at $v$;
this does not change the space nor its topology, since the cylinder $C_{p}$ was the same as $C_{pe}$.
For what concerns the piecewise-canonical homeomorphisms (which we will define in \cref{sub:groups:of:piecewise:canonical:homeomorphisms}), say that $\Gamma$ is the original graph and $\Gamma^*$ is obtained by contracting $e$:
then $\Shift(\Gamma|Y)$ is isomorphic to $\Shift(\Gamma^*|Y^*)$, where $Y^*$ is obtained from $Y$ by replacing each copy of $v=\init(e)$ and of $\term(e)$ in $Y$ with a copy of the vertex resulting from the identification of the two.

Observe that \cref{ass:no:out:degree:1} implies that every \textit{inescapable cycle} (a cycle whose every vertex has out-degree $1$) must be a single loop.
Thus, any isolated point of $\Shift(\Gamma|Y)$ is simply $p x^\infty$ (a prefix $p$ followed by infinitely many instances of the digit $x$) for some $x \in V\Gamma$ such that $\init^{-1}(x)$ consists of a single loop.

In \cite{RearrConj}, vertices of out-degree $1$ were altogether disallowed by the \textit{expanding} condition.
We will fix this with the inclusion of type 0 reductions.

\subsection{The forest of paths}

The edge shift $\Shift(\Gamma|Y)$ can be seen as the boundary of a rooted forest:

\begin{definition}
\label{def:forest:of:paths}
Each choice of $\Gamma$ and $Y$ defines a (vertex- and edge-colored, rooted) \textbf{forest of paths} $\F(\Gamma|Y)$ as follows:
\begin{itemize}
    \item The set of colors is $V\Gamma$.
    \item The set of nodes is $\Lang(\Gamma|Y)$ and each $y e_1 \dots e_k \in \Lang(\Gamma|Y)$ is colored by $\col\left(\term(e_k)\right)$, or by $y$ when $k=0$.
    The set of roots is $Y$ (i.e., when $k=0$).
    \item There is a line from a root $y$ to a node $ye$ if $e \in E\Gamma$ originates from the vertex of $\Gamma$ corresponding to $y$.
    Such line is colored by $\col\left(\term(e)\right)$.
    \item There is a line from $y e_1 \dots e_{k-1}$ to any $y e_1 \dots e_k$, colored by $\col\left(\term(e_k)\right)$.
\end{itemize}
\end{definition}

Each internal node shares its color with the line terminating at it.
In this sense the edge-coloring is redundant, but it will be useful for strand diagrams.

\begin{example}
\label{ex:full:shift}
Let $V\Gamma = \{v\}$ and $Y$ consists of a single instance of $v$.
Then $\Shift(\Gamma|Y)$ is a \textbf{full shift}, which is the set of all infinite words over the alphabet $E\Gamma$.
The forest of paths, in this case, is an $|E\Gamma|$-ary rooted tree.
\end{example}

\begin{example}
\label{ex:edge:shift}
Consider the graph $\Gamma$ shown in \cref{fig:graph} and let $Y$ consist of a copy of $B$ and one of $G$.
Then $\Lang(\Gamma|Y)$ includes $B1334$ and $G4200$, but not $G32$ nor $B22$.
A finite portion of its forest of paths is portrayed in \cref{fig:forest}.
\end{example}

\begin{figure}
\centering
\begin{tikzpicture}
    \node[vertex,red] (R) at (-2,0) {$R$};
    \node[vertex,blue] (B) at (0,0) {$B$};
    \node[vertex,Green] (G) at (2,0) {$G$};
    \draw[edge] (R) to[out=210, in=150, min distance=1cm] node[left]{$0$} (R);
    \draw[edge] (B) to[out=30, in=150] node[above]{$1$} (G);
    \draw[edge] (B) to[out=180,in=0] node[above]{$2$} (R);
    \draw[edge] (G) to[out=30, in=330, min distance=1cm] node[right]{$3$} (G);
    \draw[edge] (G) to[out=210, in=330] node[below]{$4$} (B);
\end{tikzpicture}
\caption{A graph $\Gamma$ that will be used as an example throughout the paper.}
\label{fig:graph}
\end{figure}
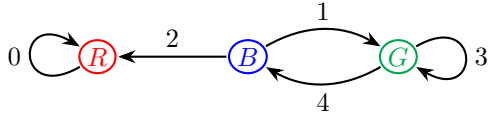

\begin{figure}
\centering
\begin{tikzpicture}[scale=.75]
    \node[node,blue] (B) at (-2,0) {};
    \node[node,Green] (G) at (2,0) {};
    \node[node,Green] (B1) at (-3,-1) {};
    \node[node,red] (B2) at (-1,-1) {};
    \node[node,Green] (G3) at (1,-1) {};
    \node[node,blue] (G4) at (3,-1) {};
    \draw[line,Green] (B) to (B1);
    \draw[line,red] (B) to (B2);
    \draw[line,Green] (G) to (G3);
    \draw[line,blue] (G) to (G4);
    \node[node,Green] (B13) at (-4,-2) {};
    \node[node,blue] (B14) at (-2.667,-2) {};
    \node[node,red] (B20) at (-1.333,-2) {};
    \node[node,Green] (G33) at (0,-2) {};
    \node[node,blue] (G34) at (1.333,-2) {};
    \node[node,Green] (G41) at (2.667,-2) {};
    \node[node,red] (G42) at (4,-2) {};
    \draw[line,Green] (B1) to (B13);
    \draw[line,blue] (B1) to (B14);
    \draw[line,red] (B2) to (B20);
    \draw[line,Green] (G3) to (G33);
    \draw[line,blue] (G3) to (G34);
    \draw[line,Green] (G4) to (G41);
    \draw[line,red] (G4) to (G42);
    \node[node,Green] (B133) at (-5,-3) {};
    \node[node,blue] (B134) at (-4.09090909,-3) {};
    \node[node,Green] (B141) at (-3.18181818,-3) {};
    \node[node,red] (B142) at (-2.27272727,-3) {};
    \node[node,red] (B200) at (-1.36363636,-3) {};
    \node[node,Green] (G333) at (-0.454545455,-3) {};
    \node[node,blue] (G334) at (0.454545455,-3) {};
    \node[node,Green] (G341) at (1.36363636,-3) {};
    \node[node,red] (G342) at (2.27272727,-3) {};
    \node[node,Green] (G413) at (3.18181818,-3) {};
    \node[node,blue] (G414) at (4.09090909,-3) {};
    \node[node,red] (G420) at (5,-3) {};
    \draw[line,Green] (B13) to (B133);
    \draw[line,blue] (B13) to (B134);
    \draw[line,Green] (B14) to (B141);
    \draw[line,red] (B14) to (B142);
    \draw[line,red] (B20) to (B200);
    \draw[line,Green] (G33) to (G333);
    \draw[line,blue] (G33) to (G334);
    \draw[line,Green] (G34) to (G341);
    \draw[line,blue] (G34) to (G342);
    \draw[line,Green] (G41) to (G413);
    \draw[line,blue] (G41) to (G414);
    \draw[line,red] (G42) to (G420);
\end{tikzpicture}
\caption{A portion of the forest of path for the edge shift defined in \cref{ex:edge:shift}.}
\label{fig:forest}
\end{figure}
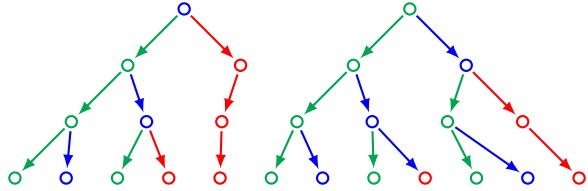

\subsection{The group of piecewise-canonical homeomorphisms}
\label{sub:groups:of:piecewise:canonical:homeomorphisms}

Let us define the elements of the groups that we wish to study.

\begin{definition}
\label{def:canonical:homeomorphism}
Let $p, q \in \Lang(\Gamma|Y)$ be nodes of $\F(\Gamma|Y)$ with the same color.
The \textbf{canonical homeomorphism} $\Phi_{p,q}$ is the homeomorphism defined as
$$\Phi_{p,q} \colon C_p \to C_q,\ p\alpha \mapsto q\alpha$$
for all infinite sequence $\alpha$ such that $p\alpha$ (and thus $q\alpha$) belongs to $\Shift(\Gamma|Y)$.
\end{definition}

\begin{definition}
\label{def:piecewise:canonical:homeomorphisms}
A homeomorphism $\phi$ of $\Shift(\Gamma|Y)$ is \textbf{piecewise-canonical} if there exist finite partitions $\{C_{w_1}, \dots, C_{w_k}\}$ and $\{C_{v_1}, \dots, C_{v_k}\}$ of $\Shift(\Gamma|Y)$ into cylinders such each restriction $\phi|_{C_{w_i}}$ is the canonical homeomorphism $\Phi_{w_i,v_i}$.
\end{definition}


It is routine to see that the set of all piecewise-canonical homeomorphisms of an edge shift $\Shift(\Gamma|Y)$ forms a group.
We will denote it by $\G(\Gamma|Y)$.

\begin{example}
\label{ex:Thompson}
Let $V\Gamma=\{v\}$, $|E\Gamma|=n$ and $Y$ consist of $k$ instances of the vertex $v$.
Then $\G(\Gamma|Y)$ is the Higman-Thompson group $V_{n,k}$.
\end{example}

The groups $\G(\Gamma|Y)$ are \textit{braided diagram groups} \cite{DiagramGroups} over digit rewriting systems, as explained (but not explored) in \cite{Order}.

\subsection{Forest pair diagrams}
\label{sub:forest:pair:diagrams}


Let $\F$ be a rooted forest.

\begin{definition}
A subforest of $\F$ is \textbf{complete} if, for each of its vertices, it either includes all or none of its children.
A subforest of $\F$ is \textbf{rooted} if it contains all the roots of $\F$.
\end{definition}

Note that each finite partition of $\Shift(\Gamma|Y)$ into cylinders corresponds to the set of leaves of a unique finite complete rooted subforest of $\F(\Gamma|Y)$.

Using this correspondence, a piecewise-canonical homeomorphism of $\Shift(\Gamma|Y)$ is determined by a \textbf{forest pair diagram}:
this is triple $(F_D,f,F_R)$, where $F_D$ and $F_R$ are finite complete rooted subforests of $\F(\Gamma|Y)$ and $f$ is a color-preserving bijection between the leaves of $F_D$ and $F_R$.
More explicitly, for each leaf $v$ of $F_D$ the homeomorphism $\phi$ represented by $(F_D,f,F_R)$ is defined as $\Phi_{v,f(v)}$ on the cylinder $C_v$.
An example is portrayed in \cref{sfig:forest:pair:diagram}.

\begin{definition}
Two forest pair diagrams are \textbf{equivalent} when they represent the same homeomorphism of $\Shift(\Gamma|Y)$.
\end{definition}

Equivalently, two forest pair diagrams are equivalent when one can be obtained from the other using the following types of expansions and reductions.

A \textbf{regular expansion} of $(F_D,f,F_R)$ consists of replacing a leaf $v$ of $F_D$ with the appropriate caret (in the unique way that produces a subforest of $\F(\Gamma|Y)$) and the leaf $f(v)$ of $F_R$ with the appropriate caret (same way as before) and replacing $f$ with a bijection that acts as $f$ on all leaves of $F_D$ other than $v$ and that maps each child $ve$ of $v$ to the child $f(v)e$ of $f(v)$.
A \textbf{regular reduction} is the inverse of a regular expansion, which thus consists of removing a caret attached at some $v$ of $F_D$ and its image in $F_R$ when $f(ve)=f(v)e$ for each leaf $ve$ of the caret.
These are the usual expansions (reductions) that occur in representatives for elements of the Higman-Thompson groups, except that the shape of the caret involved depends on the color of the parent.

A \textbf{degenerate expansion} of $(F_D,f,F_R)$ consists of removing from $F_D$ (respectively, from $F_R$) a leaf of the form $pe$ ($p$ a path and $e$ an edge) such that $C_p = C_{pe}$ (meaning that $C_p$ consists of an isolated point, or equivalently that $pe$ is the unique child of $p$) and modifying $f$ by replacing $pe$ by $p$ in its domain (respectively, range);
the resulting map is color-preserving thanks to \cref{ass:no:out:degree:1}.
A \textbf{degenerate reduction} is once again the inverse of a degenerate expansion, which thus consists of attaching a leaf $pe$ to some leaf $p$ of either $F_D$ or $F_R$ such that $C_p = C_{pe}$.
These reductions take care of the fact that, when $C_p=C_{pe}$, any canonical homeomorphism $\Phi_{p,q}$ is the same as $\Phi_{pe,q}$.

A forest pair diagram is \textbf{reduced} if no reduction can be performed on it.
Standard arguments (akin to those of \cref{prop:equivalent:semi:reduced}) show that each piecewise-canonical homeomorphism admits a unique reduced forest pair diagram.

Compositions can be computed in the standard Thompson-like way.
Given $\mathfrak{f}=(F_D,f,F_R)$ and $\mathfrak{g}=(G_D,g,G_R)$, there exists a finite complete rooted subforest $E$ of that contains both $G_D$ and $F_R$.
Then $(F_D,f,F_R)$ and $(G_D,g,D_R)$ can be expanded to $(F_D^*,f^*,E)$ and $(E,g^*,D_R^*)$, respectively, and the composition $\mathfrak{fg}$ is represented by $(F_D^*, f^* g^*, G_R^*)$.


\section{Strand diagrams for the elements of \texorpdfstring{$\G(\Gamma|Y)$}{G(Gamma|Y)}}

First developed in \cite{BM14}, strand diagrams were generalized in \cite{RearrConj} to tackle rearrangement groups, which include the groups $\G(\Gamma|Y)$.
This section introduces a simplified version of the technology of \cite{RearrConj} for the groups $\G(\Gamma|Y)$.

Recall from \cref{sub:notation:convention} that we will call \textit{points} the vertices of a strand diagram and \textit{strands} its edges, to avoid confusion with $\Gamma$ and $\F(\Gamma|Y)$.

\subsection{\texorpdfstring{$\Gamma$}{Gamma}-strand diagrams}

In this and the next subsections, we will introduce a groupoid of generalized piecewise-canonical homeomorphisms where we allow the domain and range to be edge shifts $\Shift(\Gamma|Y_1)$ and $\Shift(\Gamma|Y_2)$ with different initial multiset $Y_1$ and $Y_2$.

Let us fix a finite graph $\Gamma$.
For each vertex $v$ of $\Gamma$, we fix an arbitrary linear order of $\init^{-1}(v)$ and we write $\llbracket e_1, \ldots, e_k \rrbracket_v$ to say that $e_1 < e_2 < \dots < e_k$.

A \textbf{rotation system} of a graph is a choice of a cyclic order, for every vertex $v$, of the set $\init^{-1}(v) \cup \term^{-1}(v)$ (if there are loops, they count as two distinct elements, though such situation will not occur in our cases).
This corresponds to a planar embedding of a small neighborhood of a graph.
When a graph is equipped with a rotation system, we will denote the cyclic order at a vertex $v$ by $\llparenthesis a_1, \ldots, a_k \rrparenthesis_v$.

Below, by \textbf{split} we mean a point that is the terminus of at most one strand and the origin of at least two strands.
Conversely, by \textbf{merge} we mean a point that is the origin of at most one strand and the terminus of at least two strands.
A \textbf{source} (\textbf{sink}) is a point that is not the terminus (origin) of any strand.
An \textbf{internal split} (\textbf{internal merge}) is a split (merge) that is not a source (sink) and a \textbf{split-source} (\textbf{merge-sink}) is a split (merge) that is also a source (sink).
Finally, a point is \textbf{degenerate} if it has in- and out-degrees equal to $1$.

\begin{definition}
\label{def:gamma:strand:diagram}
A \textbf{$\Gamma$-strand diagram} is a finite acyclic graph that is vertex- and edge-colored by $V\Gamma$, is equipped with a linear order of its sources and of its sinks and with a rotation system, and such that the following conditions hold.
\begin{itemize}
    \item Every point is either a univalent source, a univalent sink, a split-source, a merge-sink, an internal split, an internal merge or is degenerate.
    \item If $v$ is a split, let $s_1, \ldots, s_k$ be the strands that originate from $v$ and, if $v$ is internal, let $s$ be the strand that terminates at $v$, so that $\llparenthesis s, s_1, \ldots, s_k \rrparenthesis_v$ or $\llparenthesis s_1, \ldots, s_k \rrparenthesis_v$.
    Let $e_1, \ldots, e_k$ be the edges of $\Gamma$ originating from $\col(v)$ so that $\llbracket e_1, \ldots, e_k \rrbracket_{\col(e)}$.
    Then $\col(s_i) = \term(e_i)$ for all $i=1,\ldots,k$ and, if $v$ is internal, $\col(v)=\col(s)$.
    \item If $v$ is a merge, let $s_1, \ldots, s_k$ be the strands that terminate at $v$ and, if $v$ is internal, let $s$ be the strand that originates from $v$, so that $\llparenthesis s, s_k, \ldots, s_1 \rrparenthesis_v$ or $\llparenthesis s_k, \ldots, s_1 \rrparenthesis_v$.
    Let $e_1, \ldots, e_k$ be the edges of $\Gamma$ originating from $\col(v)$ so that $\llbracket e_1, \ldots, e_k \rrbracket_{\col(e)}$.
    Then $\col(s_i) = \term(e_i)$ for all $i=1,\ldots,k$ and, if $v$ is internal, $\col(v)=\col(s)$.
    \item If $p$ is a degenerate point, let $s_1$ and $s_2$ be the strands that terminates at and that originates from $p$, respectively.
    Then $\col(v) = \col(s_1) = \col(s_2) = c$ and, as a vertex of $\Gamma$, $c$ is such that $\init^{-1}(c)$ consists of a single loop.
\end{itemize}
\end{definition}

The second, third and fourth conditions ensure that the neighborhoods of splits and merges look like those of nodes of the forest of paths (upside-down for merges).
For the fourth condition, compare with \cref{ass:no:out:degree:1}.

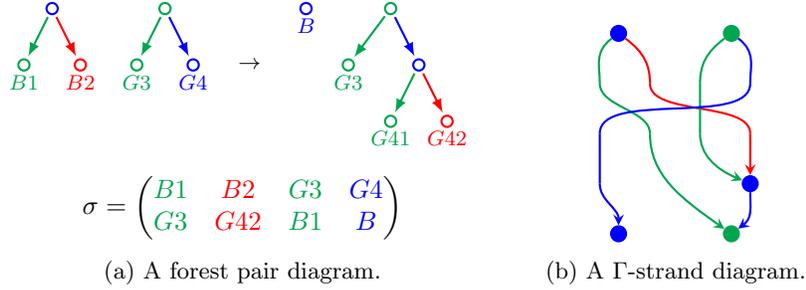
\begin{figure}
\centering
\begin{subfigure}{.55\textwidth}
\centering
\begin{tikzpicture}[scale=.75,font=\footnotesize]
    \begin{scope}[xshift=-2.5cm,xscale=.5]
    \node[node,blue] (B) at (-2,0) {};
    \node[node,Green] (G) at (2,0) {};
    \node[node,Green] (B1) at (-3,-1) {};
    \node[node,red] (B2) at (-1,-1) {};
    \node[node,Green] (G3) at (1,-1) {};
    \node[node,blue] (G4) at (3,-1) {};
    \draw[line,Green] (B) to (B1);
    \draw[line,red] (B) to (B2);
    \draw[line,Green] (G) to (G3);
    \draw[line,blue] (G) to (G4);
    \node[below,Green] at (B1) {$B1$};
    \node[below,red] at (B2) {$B2$};
    \node[below,Green] at (G3) {$G3$};
    \node[below,blue] at (G4) {$G4$};
    \end{scope}
    \node at (0,-1) {$\to$};
    \begin{scope}[xshift=1.5cm,xscale=.5]
    \node[node,blue] (B) at (-1,0) {};
    \node[node,Green] (G) at (2,0) {};
    \node[node,Green] (G3) at (.5,-1) {};
    \node[node,blue] (G4) at (3,-1) {};
    \draw[line,Green] (G) to (G3);
    \draw[line,blue] (G) to (G4);
    \node[node,Green] (G41) at (2,-2) {};
    \node[node,red] (G42) at (4,-2) {};
    \draw[line,Green] (G4) to (G41);
    \draw[line,red] (G4) to (G42);
    \node[below,blue] at (B) {$B$};
    \node[below,Green] at (G3) {$G3$};
    \node[below,Green] at (G41) {$G41$};
    \node[below,red] at (G42) {$G42$};
    \end{scope}
\end{tikzpicture}
\vspace{-.1cm}
$$\sigma =
\begin{pmatrix}
\textcolor{Green}{B1} & \textcolor{red}{B2} & \textcolor{Green}{G3} & \textcolor{blue}{G4} \\
\textcolor{Green}{G3} & \textcolor{red}{G42} & \textcolor{Green}{B1} & \textcolor{blue}{B}
\end{pmatrix}
$$
\vspace{-.4cm}
\caption{A forest pair diagram.}
\label{sfig:forest:pair:diagram}
\end{subfigure}
\begin{subfigure}{.4\textwidth}
\centering
\begin{tikzpicture}[yscale=2/3]
    \node[point,blue] (0B) at (-.75,0) {};
    \node[point,Green] (0G) at (.75,0) {};
    \coordinate (1G1) at (-1,-1) {};
    \coordinate (1R) at (-.333,-1) {};
    \coordinate (1G2) at (.333,-1) {};
    \coordinate (1B) at (1,-1) {};
    \coordinate (2B) at (-1,-2) {};
    \coordinate (2G1) at (-.333,-2) {};
    \coordinate (2G2) at (.333,-2) {};
    \coordinate (2R) at (1,-2) {};
    \coordinate (3B1) at (-1,-3) {};
    \node[point,blue] (3B2) at (1,-3) {};
    \node[point,blue] (4B) at (-.75,-4) {};
    \node[point,Green] (4G) at (.75,-4) {};
    \draw[Green] (0B) to[out=225,in=90] (1G1);
    \draw[red] (0B) to[out=315,in=90] (1R);
    \draw[Green] (0G) to[out=225,in=90] (1G2);
    \draw[blue] (0G) to[out=315,in=90] (1B);
    \draw[Green] (1G1) to[out=270,in=90] (2G1);
    \draw[red] (1R) to[out=270,in=90] (2R);
    \draw[Green] (1G2) to[out=270,in=90] (2G2);
    \draw[blue] (1B) to[out=270,in=90] (2B);
    \draw[blue] (2B) to[out=270,in=90] (3B1);
    \draw[strand,Green] (2G2) to[out=270,in=150] (3B2);
    \draw[strand,red] (2R) to (3B2);
    \draw[strand,blue] (3B1) to[out=270,in=90] (4B);
    \draw[strand,Green] (2G1) to[out=270,in=135] (4G);
    \draw[strand,blue] (3B2) to[out=270,in=45] (4G);
\end{tikzpicture}
\caption{A $\Gamma$-strand diagram.}
\label{sfig:strand:diagram}
\end{subfigure}
\caption{An element of $\G(\Gamma|Y)$, with $\Gamma$ and $Y$ from \cref{ex:edge:shift}.}
\label{fig:element}
\end{figure}

Given a $\Gamma$-strand diagram with sources $D_1,\dots,D_n$ and sinks $R_1,\dots,R_m$, its \textbf{domain} is the tuple $\left(\col(D_1),\dots,\col(D_n)\right)$ and its \textbf{range} is $\left(\col(R_1),\dots,\col(R_m)\right)$.
Domains and ranges correspond to choices of multisets $Y$ with underlying set $V\Gamma$ together with linear orders of such multisets.

\cref{sfig:strand:diagram} depicts a strand diagram with domain and range equal to $(\textcolor{blue}{B},\textcolor{Green}{G})$.

\subsection{The groupoid \texorpdfstring{$\Groupoid(\Gamma)$}{G(Gamma)}}

Let us see how to reduce and compose strand diagrams.

\subsubsection{Reductions of strand diagrams}
\label{ssub:reductions}

We say that two strand diagrams are \textbf{equivalent} if one can be obtained from the other using the two following reductions (also illustrated in \cref{fig:reductions}).
We will write $[F]$ to denote the equivalence class of a $\Gamma$-strand diagram $F$.

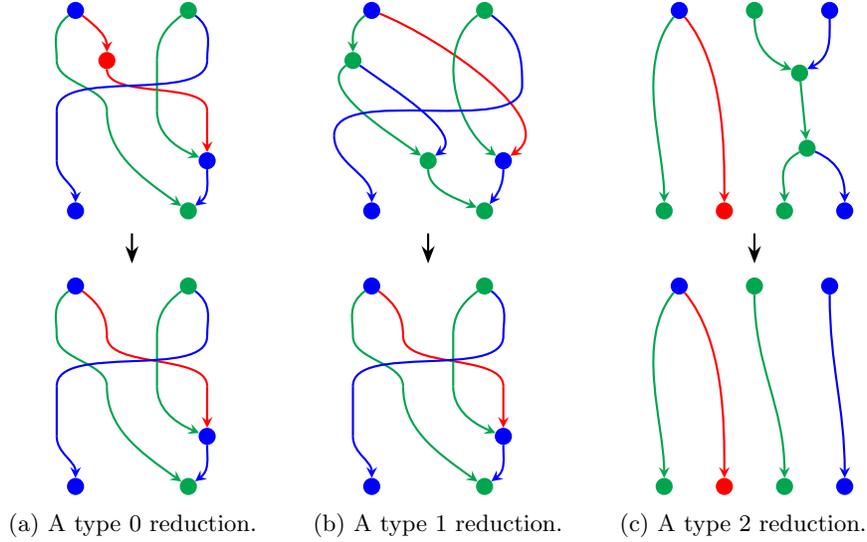
\begin{figure}
\centering
\begin{subfigure}{.3\textwidth}
\centering
\begin{tikzpicture}[yscale=2/3]
    \begin{scope}[yshift=4.75cm]
    \node[point,blue] (0B) at (-.75,0) {};
    \node[point,Green] (0G) at (.75,0) {};
    \coordinate (1G1) at (-1,-1) {};
    \node[point,red] (1R) at (-.333,-1) {};
    \coordinate (1G2) at (.333,-1) {};
    \coordinate (1B) at (1,-1) {};
    \coordinate (2B) at (-1,-2) {};
    \coordinate (2G1) at (-.333,-2) {};
    \coordinate (2G2) at (.333,-2) {};
    \coordinate (2R) at (1,-2) {};
    \coordinate (3B1) at (-1,-3) {};
    \node[point,blue] (3B2) at (1,-3) {};
    \node[point,blue] (4B) at (-.75,-4) {};
    \node[point,Green] (4G) at (.75,-4) {};
    \draw[Green] (0B) to[out=225,in=90] (1G1);
    \draw[strand,red] (0B) to[out=315,in=90] (1R);
    \draw[Green] (0G) to[out=225,in=90] (1G2);
    \draw[blue] (0G) to[out=315,in=90] (1B);
    \draw[Green] (1G1) to[out=270,in=90] (2G1);
    \draw[red] (1R) to[out=270,in=90] (2R);
    \draw[Green] (1G2) to[out=270,in=90] (2G2);
    \draw[blue] (1B) to[out=270,in=90] (2B);
    \draw[blue] (2B) to[out=270,in=90] (3B1);
    \draw[strand,Green] (2G2) to[out=270,in=150] (3B2);
    \draw[strand,red] (2R) to (3B2);
    \draw[strand,blue] (3B1) to[out=270,in=90] (4B);
    \draw[strand,Green] (2G1) to[out=270,in=135] (4G);
    \draw[strand,blue] (3B2) to[out=270,in=45] (4G);
    \end{scope}
    \draw[-Stealth] (0,.3) to (0,-.3);
    \begin{scope}[yshift=-.75cm]
    \node[point,blue] (0B) at (-.75,0) {};
    \node[point,Green] (0G) at (.75,0) {};
    \coordinate (1G1) at (-1,-1) {};
    \coordinate (1R) at (-.333,-1) {};
    \coordinate (1G2) at (.333,-1) {};
    \coordinate (1B) at (1,-1) {};
    \coordinate (2B) at (-1,-2) {};
    \coordinate (2G1) at (-.333,-2) {};
    \coordinate (2G2) at (.333,-2) {};
    \coordinate (2R) at (1,-2) {};
    \coordinate (3B1) at (-1,-3) {};
    \node[point,blue] (3B2) at (1,-3) {};
    \node[point,blue] (4B) at (-.75,-4) {};
    \node[point,Green] (4G) at (.75,-4) {};
    \draw[Green] (0B) to[out=225,in=90] (1G1);
    \draw[red] (0B) to[out=315,in=90] (1R);
    \draw[Green] (0G) to[out=225,in=90] (1G2);
    \draw[blue] (0G) to[out=315,in=90] (1B);
    \draw[Green] (1G1) to[out=270,in=90] (2G1);
    \draw[red] (1R) to[out=270,in=90] (2R);
    \draw[Green] (1G2) to[out=270,in=90] (2G2);
    \draw[blue] (1B) to[out=270,in=90] (2B);
    \draw[blue] (2B) to[out=270,in=90] (3B1);
    \draw[strand,Green] (2G2) to[out=270,in=150] (3B2);
    \draw[strand,red] (2R) to (3B2);
    \draw[strand,blue] (3B1) to[out=270,in=90] (4B);
    \draw[strand,Green] (2G1) to[out=270,in=135] (4G);
    \draw[strand,blue] (3B2) to[out=270,in=45] (4G);
    \end{scope}
\end{tikzpicture}
\caption{A type 0 reduction.}
\label{sfig:0:reduction}
\end{subfigure}
\begin{subfigure}{.35\textwidth}
\centering
\begin{tikzpicture}[yscale=2/3]
    \begin{scope}[yshift=4.75cm]
    \node[point,blue] (0B) at (-.75,0) {};
    \node[point,Green] (0G) at (.75,0) {};
    \node[point,Green] (1G) at (-1,-1) {};
    \coordinate (1B) at (1.25,-1.5) {};
    \coordinate (2B) at (-1.25,-2.5) {};
    \node[point,Green] (3G) at (0,-3) {};
    \node[point,blue] (3B) at (1,-3) {};
    \node[point,blue] (4B) at (-.75,-4) {};
    \node[point,Green] (4G) at (.75,-4) {};
    \draw[strand,Green] (0B) to[out=225,in=90] (1G);
    \draw[strand,red] (0B) to[out=330,in=45,looseness=.75] (3B);
    \draw[strand,Green] (0G) to[out=225,in=135,looseness=.75] (3B);
    \draw[blue] (0G) to[out=315,in=90] (1B);
    \draw[strand,Green] (1G) to[out=225,in=135,looseness=.75] (3G);
    \draw[strand,blue] (1G) to[out=315,in=45,looseness=.75] (3G);
    \draw[blue] (1B) to[out=270,in=90] (2B);
    \draw[strand,blue] (2B) to[out=270,in=90] (4B);
    \draw[strand,Green] (3G) to[out=270,in=150] (4G);
    \draw[strand,blue] (3B) to[out=270,in=60] (4G);
    \end{scope}
    \draw[-Stealth] (0,.3) to (0,-.3);
    \begin{scope}[yshift=-.75cm]
    \node[point,blue] (0B) at (-.75,0) {};
    \node[point,Green] (0G) at (.75,0) {};
    \coordinate (1G1) at (-1,-1) {};
    \coordinate (1R) at (-.333,-1) {};
    \coordinate (1G2) at (.333,-1) {};
    \coordinate (1B) at (1,-1) {};
    \coordinate (2B) at (-1,-2) {};
    \coordinate (2G1) at (-.333,-2) {};
    \coordinate (2G2) at (.333,-2) {};
    \coordinate (2R) at (1,-2) {};
    \coordinate (3B1) at (-1,-3) {};
    \node[point,blue] (3B2) at (1,-3) {};
    \node[point,blue] (4B) at (-.75,-4) {};
    \node[point,Green] (4G) at (.75,-4) {};
    \draw[Green] (0B) to[out=225,in=90] (1G1);
    \draw[red] (0B) to[out=315,in=90] (1R);
    \draw[Green] (0G) to[out=225,in=90] (1G2);
    \draw[blue] (0G) to[out=315,in=90] (1B);
    \draw[Green] (1G1) to[out=270,in=90] (2G1);
    \draw[red] (1R) to[out=270,in=90] (2R);
    \draw[Green] (1G2) to[out=270,in=90] (2G2);
    \draw[blue] (1B) to[out=270,in=90] (2B);
    \draw[blue] (2B) to[out=270,in=90] (3B1);
    \draw[strand,Green] (2G2) to[out=270,in=150] (3B2);
    \draw[strand,red] (2R) to (3B2);
    \draw[strand,blue] (3B1) to[out=270,in=90] (4B);
    \draw[strand,Green] (2G1) to[out=270,in=135] (4G);
    \draw[strand,blue] (3B2) to[out=270,in=45] (4G);
    \end{scope}
\end{tikzpicture}
\caption{A type 1 reduction.}
\label{sfig:1:reduction}
\end{subfigure}
\begin{subfigure}{.3\textwidth}
\centering
\begin{tikzpicture}[yscale=2/3]
    \begin{scope}[yshift=4.75cm]
    \node[point,blue] (0B1) at (-1,0) {};
    \node[point,Green] (0G) at (0,0) {};
    \node[point,blue] (0B2) at (1,0) {};
    \node[point,Green] (1G) at (.6,-1.25) {};
    \node[point,Green] (2G) at (.7,-2.75) {};
    \node[point,Green] (3G1) at (-1.2,-4) {};
    \node[point,red] (3R) at (-.4,-4) {};
    \node[point,Green] (3G2) at (.4,-4) {};
    \node[point,blue] (3B) at (1.2,-4) {};
    \draw[strand,Green] (0B1) to[out=240,in=90,looseness=.75] (3G1);
    \draw[strand,red] (0B1) to[out=300,in=90,looseness=.75] (3R);
    \draw[strand,Green] (0G) to[out=270,in=150] (1G);
    \draw[strand,blue] (0B2) to[out=270,in=35] (1G);
    \draw[strand,Green] (1G) to (2G);
    \draw[strand,Green] (2G) to[out=215,in=90] (3G2);
    \draw[strand,blue] (2G) to[out=330,in=90] (3B);
    \end{scope}
    \draw[-Stealth] (0,.3) to (0,-.3);
    \begin{scope}[yshift=-.75cm]
    \node[point,blue] (0B1) at (-1,0) {};
    \node[point,Green] (0G) at (0,0) {};
    \node[point,blue] (0B2) at (1,0) {};
    \node[point,Green] (3G1) at (-1.2,-4) {};
    \node[point,red] (3R) at (-.4,-4) {};
    \node[point,Green] (3G2) at (.4,-4) {};
    \node[point,blue] (3B) at (1.2,-4) {};
    \draw[strand,Green] (0B1) to[out=240,in=90,looseness=.75] (3G1);
    \draw[strand,red] (0B1) to[out=300,in=90,looseness=.75] (3R);
    \draw[strand,Green] (0G) to[out=270,in=90] (3G2);
    \draw[strand,blue] (0B2) to[out=270,in=90] (3B);
    \end{scope}
\end{tikzpicture}
\caption{A type 2 reduction.}
\label{sfig:2:reduction}
\end{subfigure}
\caption{Examples of the three types of reductions of $\Gamma$-strand diagrams.}
\label{fig:reductions}
\end{figure}

\textbf{Type 0:}
Assume that $v$ is a degenerate point and let $s_1$, $s_2$ be the strands terminating at and originating from $v$, respectively.
By \cref{def:gamma:strand:diagram}, they have the same color $c$.
Then replace $s_1, s_2$ and $v$ by a single $c$-colored strand $s$.

\textbf{Type 1:}
Assume that $v$ and $w$ are a $c$-colored split and a $c$-colored merge that are joined by strand $s_1, \dots, s_k$  so that the rotation systems are $\llparenthesis s^v, s_1, \dots s_k \rrparenthesis_v$ or $\llparenthesis s_1, \dots s_k \rrparenthesis_v$ (depending on whether $v$ is an internal split or a split-source) and $\llparenthesis s^w, s_k, \dots, s_1 \rrparenthesis_w$ or $\llparenthesis s_k, \dots, s_1 \rrparenthesis_w$ (depending on whether $w$ is an internal merge or a merge-sink).
Then remove $s_1, \dots, s_k$ and attach $s^v$ or $v$ and $s^w$ or $w$ (depending on whether they are internal) into a single $c$-colored strand.

\textbf{Type 2:}
Assume that $v$ and $w$ are a $c$-colored merge and a $c$-colored split that are joined by a strand $s$ and have rotation systems $\llparenthesis s^v, s^v_k, \dots s^v_1 \rrparenthesis_v$ and $\llparenthesis s^w, s^w_1, \dots, s^w_k \rrparenthesis_w$.
Then remove $v$, $w$ and $s$ and attach each strand $s^v_i$ with $s^w_i$ into a single strand.

Such reductions are well-defined thanks to the conditions of \cref{def:gamma:strand:diagram}.
Moreover, note that they do not alter the domain nor the range of diagrams.

A $\Gamma$-strand diagram is \textbf{reduced} if no reduction can be performed on it.
It is routine to check (and is proved in \cite{RearrConj}, although without type 0 reductions) that each equivalence class includes a unique reduced $\Gamma$-strand diagram.

\subsubsection{Composition of strand diagrams}
\label{sub:subgroupoid}

Let $\Groupoid(\Gamma)$ be the set of equivalence classes of $\Gamma$-strand diagrams.
Assume that $A$ is a $\Gamma$-strand diagram with domain $(a_1,\dots,a_i)$ and range $(b_1,\dots,b_j)$ and that $B$ is a $\Gamma$-strand diagram with domain $(b_1,\dots,b_j)$ and range $(c_1,\dots,c_k)$.
The composition $AB$ has diagram with domain $(a_1,\dots,a_i)$ and range $(c_1,\dots,c_k)$ obtained by gluing each sink $b_i$ of $A$ with the corresponding source $b_i$ of $B$.
For example, \cref{fig:composition} depicts the second power of the element from \cref{fig:element}.

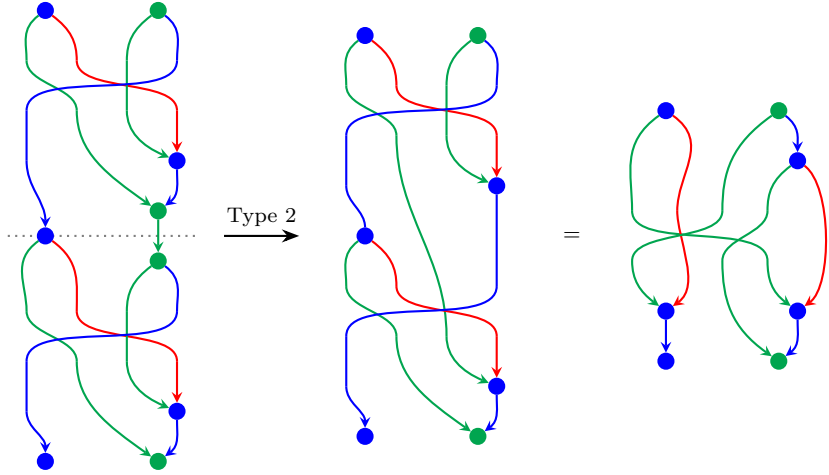
\begin{figure}
\centering
\begin{tikzpicture}[yscale=2/3,font=\footnotesize]
    \begin{scope}[xshift=-2.25cm]
    \node[point,blue] (0B) at (-.75,0) {};
    \node[point,Green] (0G) at (.75,0) {};
    \coordinate (1G1) at (-1,-1) {};
    \coordinate (1R) at (-.333,-1) {};
    \coordinate (1G2) at (.333,-1) {};
    \coordinate (1B) at (1,-1) {};
    \coordinate (2B) at (-1,-2) {};
    \coordinate (2G1) at (-.333,-2) {};
    \coordinate (2G2) at (.333,-2) {};
    \coordinate (2R) at (1,-2) {};
    \coordinate (3B1) at (-1,-3) {};
    \node[point,blue] (3B2) at (1,-3) {};
    \node[point,Green] (4G) at (.75,-4) {};
    \draw[gray,dotted] (-1.25,-4.5) to (1.25,-4.5);
    \node[point,blue] (45B) at (-.75,-4.5) {};
    \node[point,Green] (5G) at (.75,-5) {};
    \coordinate (6G1) at (-1,-6) {};
    \coordinate (6R) at (-.333,-6) {};
    \coordinate (6G2) at (.333,-6) {};
    \coordinate (6B) at (1,-6) {};
    \coordinate (7B) at (-1,-7) {};
    \coordinate (7G1) at (-.333,-7) {};
    \coordinate (7G2) at (.333,-7) {};
    \coordinate (7R) at (1,-7) {};
    \coordinate (8B1) at (-1,-8) {};
    \node[point,blue] (8B2) at (1,-8) {};
    \node[point,blue] (9B) at (-.75,-9) {};
    \node[point,Green] (9G) at (.75,-9) {};
    \draw[Green] (0B) to[out=225,in=90] (1G1);
    \draw[red] (0B) to[out=315,in=90] (1R);
    \draw[Green] (0G) to[out=225,in=90] (1G2);
    \draw[blue] (0G) to[out=315,in=90] (1B);
    \draw[Green] (1G1) to[out=270,in=90] (2G1);
    \draw[red] (1R) to[out=270,in=90] (2R);
    \draw[Green] (1G2) to[out=270,in=90] (2G2);
    \draw[blue] (1B) to[out=270,in=90] (2B);
    \draw[blue] (2B) to[out=270,in=90] (3B1);
    \draw[strand,Green] (2G2) to[out=270,in=150] (3B2);
    \draw[strand,red] (2R) to (3B2);
    \draw[strand,blue] (3B1) to[out=270,in=90] (45B);
    \draw[strand,Green] (2G1) to[out=270,in=135] (4G);
    \draw[strand,blue] (3B2) to[out=270,in=45] (4G);
    \draw[strand,Green] (4G) to (5G);
    \draw[Green] (45B) to[out=225,in=90] (6G1);
    \draw[red] (45B) to[out=315,in=90] (6R);
    \draw[Green] (5G) to[out=225,in=90] (6G2);
    \draw[blue] (5G) to[out=315,in=90] (6B);
    \draw[Green] (6G1) to[out=270,in=90] (7G1);
    \draw[red] (6R) to[out=270,in=90] (7R);
    \draw[Green] (6G2) to[out=270,in=90] (7G2);
    \draw[blue] (6B) to[out=270,in=90] (7B);
    \draw[blue] (7B) to[out=270,in=90] (8B1);
    \draw[strand,Green] (7G2) to[out=270,in=150] (8B2);
    \draw[strand,red] (7R) to (8B2);
    \draw[strand,blue] (8B1) to[out=270,in=90] (9B);
    \draw[strand,Green] (7G1) to[out=270,in=135] (9G);
    \draw[strand,blue] (8B2) to[out=270,in=45] (9G);
    \end{scope}
    \draw[-Stealth] (-.625,-4.5) to node[above]{Type 2} (.375,-4.5);
    \begin{scope}[xshift=2cm,yshift=-.5cm]
    \node[point,blue] (0B) at (-.75,0) {};
    \node[point,Green] (0G) at (.75,0) {};
    \coordinate (1G1) at (-1,-1) {};
    \coordinate (1R) at (-.333,-1) {};
    \coordinate (1G2) at (.333,-1) {};
    \coordinate (1B) at (1,-1) {};
    \coordinate (2B) at (-1,-2) {};
    \coordinate (2G1) at (-.333,-2) {};
    \coordinate (2G2) at (.333,-2) {};
    \coordinate (2R) at (1,-2) {};
    \coordinate (3B1) at (-1,-3) {};
    \node[point,blue] (3B2) at (1,-3) {};
    \node[point,blue] (4B) at (-.75,-4) {};
    \coordinate (5G1) at (-1,-5) {};
    \coordinate (5R) at (-.333,-5) {};
    \coordinate (5G2) at (.333,-5) {};
    \coordinate (5B) at (1,-5) {};
    \coordinate (6B) at (-1,-6) {};
    \coordinate (6G1) at (-.333,-6) {};
    \coordinate (6G2) at (.333,-6) {};
    \coordinate (6R) at (1,-6) {};
    \coordinate (7B1) at (-1,-7) {};
    \node[point,blue] (7B2) at (1,-7) {};
    \node[point,blue] (8B) at (-.75,-8) {};
    \node[point,Green] (8G) at (.75,-8) {};
    \draw[Green] (0B) to[out=225,in=90] (1G1);
    \draw[red] (0B) to[out=315,in=90] (1R);
    \draw[Green] (0G) to[out=225,in=90] (1G2);
    \draw[blue] (0G) to[out=315,in=90] (1B);
    \draw[Green] (1G1) to[out=270,in=90] (2G1);
    \draw[red] (1R) to[out=270,in=90] (2R);
    \draw[Green] (1G2) to[out=270,in=90] (2G2);
    \draw[blue] (1B) to[out=270,in=90] (2B);
    \draw[blue] (2B) to[out=270,in=90] (3B1);
    \draw[strand,Green] (2G2) to[out=270,in=150] (3B2);
    \draw[strand,red] (2R) to (3B2);
    \draw[blue] (3B1) to[out=270,in=90] (4B);
    \draw[Green] (2G1) to[out=270,in=90] (6G2);
    \draw[blue] (3B2) to (5B);
    \draw[Green] (4B) to[out=225,in=90] (5G1);
    \draw[red] (4B) to[out=315,in=90] (5R);
    \draw[Green] (5G1) to[out=270,in=90] (6G1);
    \draw[red] (5R) to[out=270,in=90] (6R);
    \draw[blue] (5B) to[out=270,in=90] (6B);
    \draw[blue] (6B) to[out=270,in=90] (7B1);
    \draw[strand,Green] (6G2) to[out=270,in=150] (7B2);
    \draw[strand,red] (6R) to (7B2);
    \draw[strand,blue] (7B1) to[out=270,in=90] (8B);
    \draw[strand,Green] (6G1) to[out=270,in=135] (8G);
    \draw[strand,blue] (7B2) to[out=270,in=45] (8G);
    \end{scope}
    \node at (4,-4.5) {$=$};
    \begin{scope}[xshift=6cm,yshift=-2cm]
    \node[point,blue] (0B) at (-.75,0) {};
    \node[point,Green] (0G) at (.75,0) {};
    \node[point,blue] (1B) at (1,-1) {};
    \coordinate (2G1) at (-1.2,-2) {};
    \coordinate (2R1) at (-.6,-2) {};
    \coordinate (2G2) at (0,-2) {};
    \coordinate (2G3) at (.6,-2) {};
    \coordinate (3G1) at (-1.2,-3) {};
    \coordinate (3G2) at (0,-3) {};
    \coordinate (3G3) at (.6,-3) {};
    \node[point,blue] (4B1) at (-.75,-4) {};
    \node[point,blue] (4B2) at (1,-4) {};
    \node[point,blue] (5B) at (-.75,-5) {};
    \node[point,Green] (5G) at (.75,-5) {};
    \draw[Green] (0B) to[out=225,in=90] (2G1);
    \draw[red] (0B) to[out=315,in=90] (2R1);
    \draw[Green] (0G) to[out=225,in=90] (2G2);
    \draw[strand,blue] (0G) to[out=315,in=90] (1B);
    \draw[Green] (1B) to[out=225,in=90] (2G3);
    \draw[strand,red] (1B) to[out=315,in=45,looseness=.5] (4B2);
    \draw[Green] (2G1) to[out=270,in=90] (3G3);
    \draw[strand,red] (2R1) to[out=270,in=45,looseness=.75] (4B1);
    \draw[Green] (2G2) to[out=270,in=90] (3G1);
    \draw[Green] (2G3) to[out=270,in=90] (3G2);
    \draw[strand,Green] (3G1) to[out=270,in=135] (4B1);
    \draw[strand,Green] (3G3) to[out=270,in=135] (4B2);
    \draw[strand,blue] (4B1) to[out=270,in=90] (5B);
    \draw[strand,Green] (3G2) to[out=270,in=135] (5G);
    \draw[strand,blue] (4B2) to[out=270,in=45] (5G);
    \end{scope}
\end{tikzpicture}
\caption{A composition of $\Gamma$-strand diagrams, with application of reductions.}
\label{fig:composition}
\end{figure}

The diagram $AB$ produced in this way is a $\Gamma$-strand diagram.
This composition is well-defined (i.e., if $[A]=[A']$ and $[B]=[B']$ then $[AB]=[A'B']$) and it is associative and has inverses, so it defines a groupoid structure on $\Groupoid(\Gamma)$.

\subsubsection{Generators of the groupoid}
\label{ssub:groupoid:generators}

We define three types of minimal diagrams that will be useful in \cref{sub:similarities}.

A \textbf{permutation diagram} is a $\Gamma$-strand diagram with no split, no merge and no degenerate points.
Such a diagram simply permutes multisets over $V\Gamma$.

A \textbf{split diagram} is a $\Gamma$-strand diagram consisting solely of splits surrounded by straight strands.
Such a diagram simply appends carets.

A \textbf{merge diagram} is a diagram consisting of merges surrounded solely by straight strands.
Such a diagram simply appends inverted carets.
Each merge diagram is the inverse of a split diagram and viceversa.

Every $\Gamma$-strand diagram can be sliced into pieces of these types, so:

\begin{proposition}
The groupoid $\Groupoid(\Gamma)$ is generated by the (infinite) set of equivalence classes of all permutation diagrams and split diagrams.
\end{proposition}

For example, the diagram depicted in \cref{sfig:strand:diagram} can be decomposed as the composition of a split diagram (with two splits), a permutation diagram and two merge diagrams.
Although unneeded here, any $\Gamma$-strand diagram can be reduced and written as a composition of split diagrams followed by a permutation diagram followed by merge diagrams (see \cite[Proposition 2.8]{RearrConj}).

\subsection{The group \texorpdfstring{$\G(\Gamma|Y)$}{G(Gamma|Y)} as a subgroupoid of \texorpdfstring{$\Groupoid(\Gamma)$}{G(Gamma)}}

Let $(F_D,f,F_R)$ be a forest pair diagram for an element of $\G(\Gamma|Y)$.
Fix a linear order of $Y$, which translates to a linear order of the roots of $F_D$ and $F_R$.
One obtains a $\Gamma$-strand diagram by drawing an upside-down copy of $F_R$ below $F_D$ and gluing each leaf $l$ of $F_D$ with the leaf $f(l)$ of $F_R$.
Every $\Gamma$-strand diagram obtained in this way has domain and range both equal to $(y_1, \dots, y_k)$, corresponding to the linear order of $Y=\{y_1,\dots,y_k\}$.

Conversely, every reduced $\Gamma$-strand diagram can be cut uniquely into two forests $F_D$, corresponding to the splits, and $F_R$, corresponding to the merges (see \cite[Lemma 2.7]{RearrConj}).
If the domain and range of the $\Gamma$-strand diagram both equal $(y_1,\dots,y_k)$, then $F_D$ and $F_R$ are finite complete subforests of $\F(\Gamma|Y)$.
The leaves of $F_D$ and $F_R$ are ordered according to the rotation system of the $\Gamma$-strand diagram.
If $f$ denotes the bijection between the leaves of $F_D$ and those of $F_R$, then $(F_D,f,F_R)$ is a forest pair diagram for an element of $\G(\Gamma|Y)$.

Once a linear order of $Y$ (thus of the roots of $\F(\Gamma|Y)$) is fixed, these two operations (gluing forest pair diagrams; cutting reduced $\Gamma$-strand diagrams) are the inverse of each other.
One can check that compositions of forest pair diagrams correspond to compositions of $\Gamma$-strand diagrams.
Type 0 and 1 reductions of $\Gamma$-strand diagrams (\cref{ssub:reductions}) correspond to degenerate and regular reductions of forest pair diagrams (\cref{sub:forest:pair:diagrams}), while type 2 reductions allow to compute subforests that contain the domain forest of the first element and the range forest of the second element.
Ultimately, we have the following fact.

\begin{proposition}
\label{prop:TFGoES:is:subgroupoid}
Fix a linear order $y_1 < \dots < y_k$ of the elements of a multiset $Y$ with underlying set $V\Gamma$.
The group $\G(\Gamma|Y)$ is isomorphic to the subgroupoid of $\Groupoid(\Gamma)$ of all elements whose domain and range are both $(y_1,\dots,y_k)$.
\end{proposition}


\section{Closed strand diagrams and conjugacy classes}

In this section we will see how transformations of certain closed diagrams encode conjugacy in the groupoid.
This will be used in the following section to solve the problem of deciding whether, given two $\Gamma$-strand diagrams with the same domain and range, they represent conjugate elements of $\Groupoid(\Gamma)$.

\begin{remark}
By \cref{prop:TFGoES:is:subgroupoid}, solving the conjugacy problem in the groupoid $\Groupoid(\Gamma)$ allows to solve it in the group $\G(\Gamma|Y)$.
Indeed, if $g,f \in \Groupoid(\Gamma)$ have domain and range both corresponding to a linear order of some multiset $Y$ (so $x,y \in \G(\Gamma|Y)$) and if $z \in \Groupoid(\Gamma)$ is such that $g = h f h^{-1}$, the domain and range of $h$ must be the same as those of $g$ and $f$, meaning that $h \in \G(\Gamma|Y)$.
\end{remark}

\subsection{Closing a strand diagram}

Consider a $\Gamma$-strand diagram whose domain and range are the same (equivalently, it represents an element of some group $\G(\Gamma|Y)$).
Such diagram can be \textit{closed} by identifying each source with the corresponding sink;
when we identify a split with a merge, we add a strand joining them.
A \textbf{$\Gamma$-closed strand diagram} is the result of such operation, see for example \cref{fig:closed:strand:diagram}.

We let $\C$ denote the map sending a $\Gamma$-strand diagram to its $\Gamma$-closed strand diagram.
This is a bijection between the set of $\Gamma$-strand diagrams with equal domain and range and the set of $\Gamma$-closed strand diagrams.

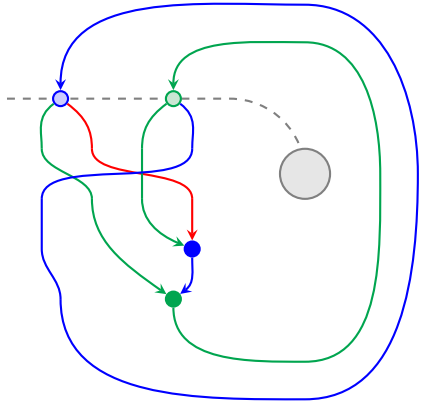
\begin{figure}
\centering
\begin{tikzpicture}
    \draw[gray,dashed] (0,0) to[out=90,in=0] (-1,1) to (-4,1);
    \draw[gray,fill=gray!20] (0,0) circle (.333cm);
    \begin{scope}[xshift=-2.5cm,yshift=1cm,yscale=2/3]
    \node[point,blue,fill=blue!20] (0B) at (-.75,0) {};
    \node[point,Green,fill=Green!20] (0G) at (.75,0) {};
    \coordinate (1G1) at (-1,-1) {};
    \coordinate (1R) at (-.333,-1) {};
    \coordinate (1G2) at (.333,-1) {};
    \coordinate (1B) at (1,-1) {};
    \coordinate (2B) at (-1,-2) {};
    \coordinate (2G1) at (-.333,-2) {};
    \coordinate (2G2) at (.333,-2) {};
    \coordinate (2R) at (1,-2) {};
    \coordinate (3B1) at (-1,-3) {};
    \node[point,blue] (3B2) at (1,-3) {};
    \coordinate (4B) at (-.75,-4) {};
    \node[point,Green] (4G) at (.75,-4) {};
    \draw[Green] (0B) to[out=225,in=90] (1G1);
    \draw[red] (0B) to[out=315,in=90] (1R);
    \draw[Green] (0G) to[out=225,in=90] (1G2);
    \draw[blue] (0G) to[out=315,in=90] (1B);
    \draw[Green] (1G1) to[out=270,in=90] (2G1);
    \draw[red] (1R) to[out=270,in=90] (2R);
    \draw[Green] (1G2) to[out=270,in=90] (2G2);
    \draw[blue] (1B) to[out=270,in=90] (2B);
    \draw[blue] (2B) to[out=270,in=90] (3B1);
    \draw[strand,Green] (2G2) to[out=270,in=150] (3B2);
    \draw[strand,red] (2R) to (3B2);
    \draw[blue] (3B1) to[out=270,in=90] (4B);
    \draw[strand,Green] (2G1) to[out=270,in=135] (4G);
    \draw[strand,blue] (3B2) to[out=270,in=45] (4G);
    \end{scope}
    \draw[strand,Green] (4G) to[out=270,in=180] (0,-2.5) to[out=0,in=270] (1,0) to[out=90,in=0] (0,1.75) to[out=180,in=90] (0G);
    \draw[strand,blue] (4B) to[out=270,in=180] (0,-3) to[out=0,in=270] (1.5,0) to[out=90,in=0] (0,2.25) to[out=180,in=90] (0B);
\end{tikzpicture}
\caption{The $\Gamma$-closed strand diagram for the diagram depicted in \cref{sfig:strand:diagram}.}
\label{fig:closed:strand:diagram}
\end{figure}

The \textbf{base line} of the $\Gamma$-closed strand diagram $\C(F)$ is the the domain and range (which are the same) of the underlying $\Gamma$-strand diagram $F$;
this is a linear order of a multiset $Y$ with underlying set $V\Gamma$.
The points of the base line are referred to as \textbf{base points}.
In drawings, the base line is represented as a dashed line that goes through the diagram at the base points, in their order, terminating at a ``hole'' in the middle of the diagram.

In addition to base points, we reprise the names of splits, merges and degenerate points from $\Gamma$-strand diagrams.
Note that points of in- and out-degree $1$ of $\C(F)$ are either base points or they are degenerate points of $F$.

\subsection{Similarities of closed strand diagrams}
\label{sub:similarities}

We define two types of transformations of $\Gamma$-closed strand diagrams that involve the base line.
Together, such transformations are called \textbf{similarities} and they will essentially allow us to forget about the base line.

\subsubsection{Base line shifts of closed strand diagrams}

Assume that a $\Gamma$-closed strand diagram $\C(F)$ has base line $(y_1, \dots, y_k)$.
Pictorially, base line shifts (which we define right below) correspond to moving the base line through a split or a merge.
For example, the diagram depicted in \cref{fig:base:line:shift} is the result of an expanding base line shift of the one depicted in \cref{fig:closed:strand:diagram}.
The converse transformation is a reducing base line shift.

Suppose that in $F$ a source $y_i$ is immediately followed by a split or that a sink $y_i$ is immediately preceded by a merge.
An \textbf{expanding base line shift} consists of moving the split (respectively, merge) from below the source $y_i$ (above the sink $y_i$) to below the corresponding sink $y_i$ (above the corresponding source $y_i$), obtaining a strand diagram $F^*$, and then taking $\C(F^*)$.
Note that $\C(F^*)$ is still a $\Gamma$-closed strand diagram, since the source and the sink corresponding to $y_i$ have the same color, and it has more points on the base line than before.

Suppose that in $F$ consecutive sources $y_i, \dots, y_{i+k}$ are all of the predecessors of a merge or that consecutive sinks $y_i, \dots, y_{i+k}$ are all of the successors of a split.
A \textbf{reducing base line shift} consists of moving the merge (respectively, split) from below the sources $y_i, \dots, y_{i+1}$ (above the sinks $y_i, \dots, y_{i+k}$) to below the corresponding sinks (above the corresponding sources), obtaining a strand diagram $F^*$ and then taking $\C(F^*)$.
The result is a $\Gamma$-closed strand diagram, this time with less points on its base line than before.

\begin{figure}
\centering
\begin{tikzpicture}
    \draw[gray,dashed] (0,0) to[out=90,in=0] (-1,1) to (-4,1);
    \draw[gray,fill=gray!20] (0,0) circle (.333cm);
    \begin{scope}[xshift=-2.5cm,yshift=1cm,yscale=2/3]
    \node[point,Green] (aG) at (.75,1) {};
    \node[point,blue,fill=blue!20] (0B1) at (-.75,0) {};
    \node[point,Green,fill=Green!20] (0G) at (.125,0) {};
    \node[point,blue,fill=blue!20] (0B2) at (1,0) {};
    \coordinate (1G1) at (-1,-1) {};
    \coordinate (1R) at (-.333,-1) {};
    \coordinate (1G2) at (.333,-1) {};
    \coordinate (1B) at (1,-1) {};
    \coordinate (2B) at (-1,-2) {};
    \coordinate (2G1) at (-.333,-2) {};
    \coordinate (2G2) at (.333,-2) {};
    \coordinate (2R) at (1,-2) {};
    \coordinate (3B1) at (-1,-3) {};
    \node[point,blue] (3B2) at (1,-3) {};
    \coordinate (4B) at (-.75,-4) {};
    \node[point,Green] (4G) at (.75,-4) {};
    \draw[Green] (0B1) to[out=225,in=90] (1G1);
    \draw[red] (0B1) to[out=315,in=90] (1R);
    \draw[Green] (0G) to[out=270,in=90] (1G2);
    \draw[blue] (0B2) to[out=270,in=90] (1B);
    \draw[Green] (1G1) to[out=270,in=90] (2G1);
    \draw[red] (1R) to[out=270,in=90] (2R);
    \draw[Green] (1G2) to[out=270,in=90] (2G2);
    \draw[blue] (1B) to[out=270,in=90] (2B);
    \draw[blue] (2B) to[out=270,in=90] (3B1);
    \draw[strand,Green] (2G2) to[out=270,in=150] (3B2);
    \draw[strand,red] (2R) to (3B2);
    \draw[blue] (3B1) to[out=270,in=90] (4B);
    \draw[strand,Green] (2G1) to[out=270,in=135] (4G);
    \draw[strand,blue] (3B2) to[out=270,in=45] (4G);
    \end{scope}
    \draw[strand,Green] (4G) to[out=270,in=180] (0,-2.5) to[out=0,in=270] (1,0) to[out=90,in=0] (0,2.25) to[out=180,in=90] (aG);
    \draw[strand,blue] (4B) to[out=270,in=180] (0,-3) to[out=0,in=270] (1.5,0) to[out=90,in=0] (0,2.75) to[out=180,in=90,looseness=1.25] (0B1);
    \draw[strand,Green] (aG) to[out=190,in=90] (0G);
    \draw[strand,blue] (aG) to[out=320,in=90] (0B2);
\end{tikzpicture}
\caption{A base line shift of the $\Gamma$-closed strand diagram from \cref{fig:closed:strand:diagram}.}
\label{fig:base:line:shift}
\end{figure}
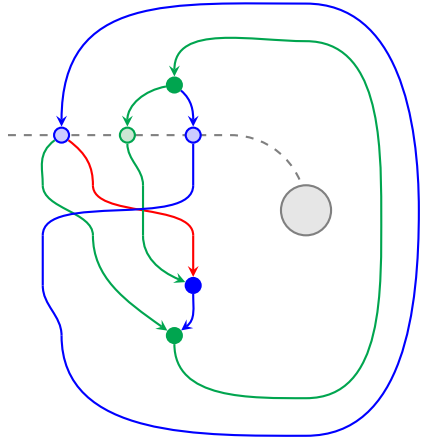

\begin{remark}
\label{rmk:shift:is:conjugacy}
If $\C(F)$ is a $\Gamma$-closed strand diagram and $\C(F^*)$ is the result of a base line shift, then $F^* = G F G^{-1}$ for a $\Gamma$-strand diagram $G$ that is a split diagram if the base line shift is expanding and a merge diagram if the base line shift is reducing.
In particular, $F$ and $F^*$ are conjugate in $\Groupoid(\Gamma)$.
\end{remark}

\subsubsection{Base line permutations of closed strand diagrams}

Consider a $\Gamma$-closed strand diagram.
As portrayed \cref{fig:permutation}, a \textbf{base line permutation} is the result of permuting the base points of such diagram.

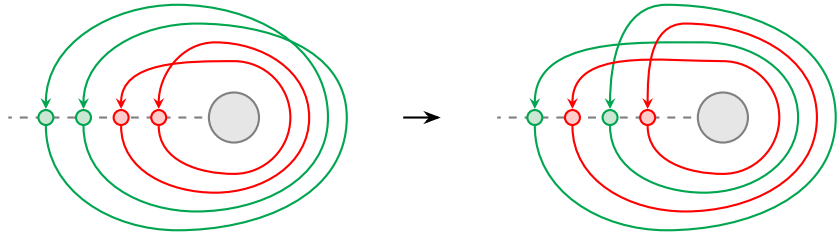
\begin{figure}
\centering
\begin{tikzpicture}
    \begin{scope}[xshift=-3.25cm]
    \draw[gray,dashed] (0,0) to (-3,0);
    \draw[gray,fill=gray!20] (0,0) circle (.333cm);
    \node[point,Green,fill=Green!20] (G1) at (-2.5,0) {};
    \node[point,Green,fill=Green!20] (G2) at (-2,0) {};
    \node[point,red,fill=red!20] (R1) at (-1.5,0) {};
    \node[point,red,fill=red!20] (R2) at (-1,0) {};
    \draw[strand,Green] (G1) to[out=270,in=180] (-.75,-1.5) to[out=0,in=270] (1.5,0) to[out=90,in=0] (-.5,1.25) to[out=180,in=90] (G2);
    \draw[strand,Green] (G2) to[out=270,in=180] (-.5,-1.25) to[out=0,in=270] (1.25,0) to[out=90,in=0] (-.75,1.5) to[out=180,in=90] (G1);
    \draw[strand,red] (R1) to[out=270,in=180] (-.25,-1) to[out=0,in=270] (1,0) to[out=90,in=0] (-.25,1) to[out=180,in=90] (R2);
    \draw[strand,red] (R2) to[out=270,in=180] (0,-.75) to[out=0,in=270] (.75,0) to[out=90,in=0] (0,.75) to[out=180,in=90] (R1);
    \end{scope}
    \draw[-Stealth] (-1,0) to (-.5,0);
    \begin{scope}[xshift=3.25cm]
    \draw[gray,dashed] (0,0) to (-3,0);
    \draw[gray,fill=gray!20] (0,0) circle (.333cm);
    \node[point,Green,fill=Green!20] (G1) at (-2.5,0) {};
    \node[point,red,fill=red!20] (R1) at (-2,0) {};
    \node[point,Green,fill=Green!20] (G2) at (-1.5,0) {};
    \node[point,red,fill=red!20] (R2) at (-1,0) {};
    \draw[strand,Green] (G1) to[out=270,in=180] (-.75,-1.5) to[out=0,in=270] (1.5,0) to[out=90,in=0] (-.75,1.5) to[out=180,in=90] (G2);
    \draw[strand,Green] (G2) to[out=270,in=180] (-.25,-1) to[out=0,in=270] (1,0) to[out=90,in=0] (-.25,1) to[out=180,in=90] (G1);
    \draw[strand,red] (R1) to[out=270,in=180] (-.5,-1.25) to[out=0,in=270] (1.25,0) to[out=90,in=0] (-.5,1.25) to[out=180,in=90] (R2);
    \draw[strand,red] (R2) to[out=270,in=180] (0,-.75) to[out=0,in=270] (.75,0) to[out=90,in=0] (0,.75) to[out=180,in=90] (R1);
    \end{scope}
\end{tikzpicture}
\caption{A base line permutation of a $\Gamma$-closed strand diagram.}
\label{fig:permutation}
\end{figure}

\begin{remark}
\label{rmk:permutation:is:conjugacy}
If $\C(F)$ is a $\Gamma$-closed strand diagram and $\C(F^*)$ is the result of a base line permutation, then $F^* = P F P^{-1}$ for a permutation diagram $P$.
In particular, $F$ and $F^*$ are conjugate in $\Groupoid(\Gamma)$.
\end{remark}

\subsection{Reductions of closed strand diagrams}
\label{sub:type:3:reduction}

A $\Gamma$-closed strand diagram can also be transformed with \textbf{reductions} of four types, which we describe below.

\subsubsection{Reductions of type 0, 1 and 2}

Type 0, 1 and 2 reductions of a $\Gamma$-closed strand diagram $\C(F)$ are the same as those described in \cref{ssub:reductions} for the underlying $\Gamma$-strand diagram $F$.

\begin{remark}
\label{rmk:equivalent:open:then:equivalent:closed}
Reductions of these types do not modify the element of $\Groupoid(\Gamma)$ represented by $F$.
Thus, $[F]=[F^*]$ if and only if $\C(F)$ can be transformed into $\C(F^*)$ using reductions of type 0, 1 and 2.
\end{remark}

Such reductions often need to be ``unlocked'' by a base line shift, as the base line may lie between a merge and a split (see, for example, \cref{fig:closed:strand:diagram,,fig:base:line:shift}).

\subsubsection{Reductions of type 3}

Let us say that a \textbf{loop} in a $\Gamma$-closed strand diagram is a cycle whose points all belong to the base line.
Note that the points of a loop all share the same color.
Although not necessary for our purposes, one can show that loops of $\C(F)$ correspond to cylinders that are pointwise periodic under $F$.

For the following definition, refer to the example in \cref{fig:3:reduction}.

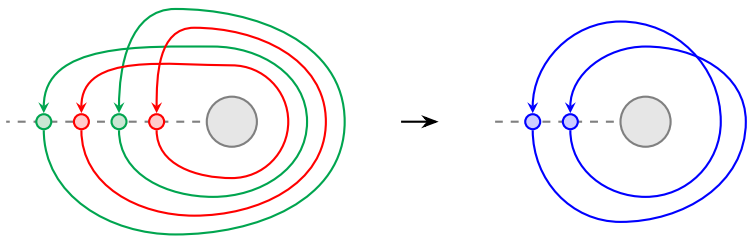
\begin{figure}
\centering
\begin{tikzpicture}
    \begin{scope}[xshift=-3.25cm]
    \draw[gray,dashed] (0,0) to (-3,0);
    \draw[gray,fill=gray!20] (0,0) circle (.333cm);
    \node[point,Green,fill=Green!20] (G1) at (-2.5,0) {};
    \node[point,red,fill=red!20] (R1) at (-2,0) {};
    \node[point,Green,fill=Green!20] (G2) at (-1.5,0) {};
    \node[point,red,fill=red!20] (R2) at (-1,0) {};
    \draw[strand,Green] (G1) to[out=270,in=180] (-.75,-1.5) to[out=0,in=270] (1.5,0) to[out=90,in=0] (-.75,1.5) to[out=180,in=90] (G2);
    \draw[strand,Green] (G2) to[out=270,in=180] (-.25,-1) to[out=0,in=270] (1,0) to[out=90,in=0] (-.25,1) to[out=180,in=90] (G1);
    \draw[strand,red] (R1) to[out=270,in=180] (-.5,-1.25) to[out=0,in=270] (1.25,0) to[out=90,in=0] (-.5,1.25) to[out=180,in=90] (R2);
    \draw[strand,red] (R2) to[out=270,in=180] (0,-.75) to[out=0,in=270] (.75,0) to[out=90,in=0] (0,.75) to[out=180,in=90] (R1);
    \end{scope}
    \draw[-Stealth] (-1,0) to (-.5,0);
    \begin{scope}[xshift=2.25cm]
    \draw[gray,dashed] (0,0) to (-2,0);
    \draw[gray,fill=gray!20] (0,0) circle (.333cm);
    \node[point,blue,fill=blue!20] (B1) at (-1.5,0) {};
    \node[point,blue,fill=blue!20] (B2) at (-1,0) {};
    \draw[strand,blue] (B1) to[out=270,in=180] (-.333,-1.333) to[out=0,in=270] (1.333,0) to[out=90,in=0] (0,1) to[out=180,in=90] (B2);
    \draw[strand,blue] (B2) to[out=270,in=180] (0,-1) to[out=0,in=270] (1,0) to[out=90,in=0] (-.333,1.333) to[out=180,in=90] (B1);
    \end{scope}
\end{tikzpicture}
\caption{A type 3 reduction, based on the graph $\Gamma$ from \cref{ex:edge:shift}.}
\label{fig:3:reduction}
\end{figure}

\textbf{Type 3:}
Assume that a $\Gamma$-closed strand diagram has a consecutive sequence of base points $(y_{i+1}, y_{i+2} \dots, y_{i+kd})$ such that, for all $j \in \{1,\dots,d\}$, there is a loop with points $y_{i+j}, y_{i+j+d} \dots y_{i+j+(k-1)d}$ (in this order).
Further assume that, if the loops are colored by $v_1, \dots, v_d \in V\Gamma$, there exists a vertex $v$ of $\Gamma$ such that $\llbracket e_1, \dots e_k \rrbracket_v$ and $\term(e_j)=v_j$.
Then replace all these loops and their points with a single $v$-colored loop with new points $y'_{i+1}, \dots, y'_{i+d}$ (in this order).

Such reductions are often ``unlocked'' by a base line permutation, as the base points need to be correctly ordered (see, for example, \cref{fig:permutation,,fig:3:reduction}).

\begin{remark}
\label{rmk:reduction:is:conjugacy}
If $\C(F)$ is a $\Gamma$-closed strand diagram and $\C(F^*)$ is the result of a reduction, then $F^* = G F G^{-1}$.
The $\Gamma$-strand diagram $G$ is trivial when the reduction is of type 0, 1 or 2 (thus $F$ and $F^*$ are equivalent) and, if the reduction is of type 3, then $G$ is a composition of split diagrams (the amount of which is the winding number denoted by $k$ in \cref{sub:type:3:reduction}).
In particular, whatever the type of reduction, $F$ and $F^*$ are conjugate in $\Groupoid(\Gamma)$.
\end{remark}

\subsection{Solving the conjugacy problem}
\label{sub:algorithm}

Let us describe our tools for tackling the conjugacy problem.

\subsubsection{Equivalence of closed strand diagrams}

Two $\Gamma$-closed strand diagrams are \textbf{equivalent} if one can be obtained from the other using similarities, reductions or the inverse of reductions.

\begin{proposition}
\label{prop:conjugacy:is:equivalence}
Let $F$ and $G$ be $\Gamma$-strand diagrams.
Then $[F]$ and $[G]$ are conjugate in $\Groupoid(\Gamma)$ if and only if $\C(F)$ and $\C(G)$ are equivalent.
\end{proposition}

\begin{proof}
\cref{rmk:shift:is:conjugacy,rmk:permutation:is:conjugacy,rmk:reduction:is:conjugacy} imply that, if the closed diagrams are equivalent, then the group elements are conjugate.
For the converse, the proof follows the same strategy as that of \cite[Proposition 4.1]{RearrConj}.
Assume that $[F]=[H][G][H]^{-1}$.
Denoting by $H^{-1}$ the $\Gamma$-strand diagram obtained by flipping $H$ upside-down, we have $[F]=[H G H^{-1}]$, so $\C(F)$ and $\C(H G H^{-1})$ are equivalent by \cref{rmk:equivalent:open:then:equivalent:closed}.
Note that $\C(H G H^{-1})$ is equivalent to $\C(G)$, since the copies of $H$ and $H^{-1}$ cancel out after shifting the base line through $H$ or through $H^{-1}$ and applying reductions of type 0, 1 and 2.
Thus, $\C(F)$ is equivalent to $\C(H G H^{-1})$ which is equivalent to $\C(G)$, as needed.
\end{proof}

\subsubsection{Partial confluence for reductions of type 0, 1 and 2}

Let us say that a $\Gamma$-closed strand diagram is \textbf{semi-reduced} if no reduction of type 0, 1 and 2 can be performed on it nor on any diagram in its similarity class.
This just means that there is no reduction of type 0, 1 and 2 that is unlocked by permuting the base points and shifting the base line.

Every semi-reduced $\Gamma$-closed strand diagram decomposes into two disjoint parts:
the \textbf{split-merge part}, comprised of those connected components that feature at least a split or a merge, and the \textbf{loop part}, comprised of those connected components that only feature base points (there are no degenerate points since the diagram is semi-reduced).
Note that this is essentially the same decomposition described in \cite{GofferLederle} for almost automorphisms of trees.

We say that two $\Gamma$-closed strand diagrams are \textbf{3-equivalent} when one can be turned into the other using similarities, reductions of type 3 and their inverse.

\begin{proposition}
\label{prop:equivalent:semi:reduced}
If $\C(F)$ and $\C(G)$ are semi-reduced and equivalent, then they are 3-equivalent.
\end{proposition}

\begin{proof}
Note that the split-merge part of any two 3-equivalent diagrams are the same up to similarity.
Consider the rewriting system on the set of 3-equivalence classes determined by the reductions of type 0, 1 and 2.
Note that a reduced object for this rewriting system corresponds to all equivalent semi-reduced $\Gamma$-closed strand diagrams.
We want to show that each 3-equivalence class can be rewritten to a unique reduced 3-equivalence class.
By Newman's diamond Lemma \cite{NewmanDiamond}, it suffices to show that the rewriting system is \textit{terminating} (there is no infinite sequence of rewritings) and \textit{locally confluent} (given rewritings $X \rightarrow A$ and $X \rightarrow B$ of the same object $X$, there exist sequences of rewritings from $A$ and from $B$ to a common object $Y$).

Each reduction decreases the number of splits, merges or degenerate points, so the rewriting system is terminating.
For local confluence, assume that there are two distinct rewritings $X \rightarrow A$ and $X \rightarrow B$.
If the rewriting are \textit{disjoint} (i.e., they involve disjoint subgraphs), then clearly each rewriting can be applied after the other, producing a common object $Y$.
Assume thus that the rewritings are not disjoint.
Let us distinguish between two cases:

\textit{Case 1:}
at least one of the two rewritings is a type 0 reduction.
Then the rewritings can be performed in whichever order without changing the result.

\textit{Case 2:}
the two rewritings are a reduction of type 1 and one of type 2.
This case is portrayed in \cref{fig:confluence}.
In this case, the two rewritings produce different sets of loops which are equivalent up to a type 3 reduction, so they represent the same object of the rewriting system.

In both cases, local confluence is verified, so Newman's diamond Lemma applies and the rewriting system has unique reduced objects, as needed.
\end{proof}

\begin{figure}
\centering
\begin{tikzpicture}[font=\footnotesize]
    \begin{scope}[yshift=4.5cm]
    \begin{scope}[xshift=-3.25cm]
    \draw[gray,dashed] (0,0) to[out=225,in=0] (-.75,-.25) to (-2,-.25);
    \draw[gray,fill=gray!20] (0,0) circle (.333cm);
    \node[point,blue] (B1) at (-1.25,.375) {};
    \node[point,Green,fill=Green!20] (G) at (-1.5,-.25) {};
    \node[point,red,fill=red!20] (R) at (-1,-.25) {};
    \node[point,blue] (B2) at (-1.25,-.875) {};
    \draw[strand,Green] (B1) to[out=225,in=90] (G);
    \draw[strand,red] (B1) to[out=315,in=90] (R);
    \draw[strand,Green] (G) to[out=270,in=135] (B2);
    \draw[strand,red] (R) to[out=270,in=45] (B2);
    \draw[strand,blue] (B2) to[out=270,in=180] (-.25,-1.25) to[out=0,in=270] (.75,0) to[out=90,in=0] (-.25,1) to[out=180,in=90] (B1);
    \end{scope}
    \draw[to-to] (-1.25,0) to node[above]{Shift} (-.25,0);
    \begin{scope}[xshift=2.75cm,yshift=-.125cm]
    \draw[gray,dashed] (0,0) to[out=135,in=0] (-.75,.5) to (-1.75,.5);
    \draw[gray,fill=gray!20] (0,0) circle (.333cm);
    \node[point,blue,fill=blue!20] (B) at (-1.25,.5) {};
    \node[point,blue] (B1) at (-1.25,0) {};
    \node[point,blue] (B2) at (-1.25,-.75) {};
    \draw[strand,Green] (B1) to[out=210,in=150] (B2);
    \draw[strand,red] (B1) to[out=330,in=30] (B2);
    \draw[strand,blue] (B2) to[out=270,in=180] (-.25,-1.125) to[out=0,in=270] (.75,0) to[out=90,in=0] (-.25,1.125) to[out=180,in=90] (B);
    \draw[strand,blue] (B) to (B1);
    \end{scope}
    \end{scope}
    \draw[-Stealth] (-3.5,2.75) to node[above,rotate=90]{Type 2} (-3.5,1.75);
    \draw[-Stealth] (2.5,2.75) to node[above,rotate=270]{Type 1} (2.5,1.75);
    \begin{scope}[xshift=-3.5cm]
    \draw[gray,dashed] (0,0) to (-1.667,0);
    \draw[gray,fill=gray!20] (0,0) circle (.333cm);
    \node[point,Green,fill=Green!20] (G) at (-1.25,0) {};
    \node[point,red,fill=red!20] (R) at (-.75,0) {};
    \draw[strand,Green] (G) to[out=270,in=180] (0,-1.25) to[out=0,in=270] (1.25,0) to[out=90,in=0] (0,1.25) to[out=180,in=90] (G);
    \draw[strand,red] (R) to[out=270,in=180] (0,-.75) to[out=0,in=270] (.75,0) to[out=90,in=0] (0,.75) to[out=180,in=90] (R);
    \end{scope}
    \draw[-Stealth] (-1.25,0) to node[above]{Type 3} (-.25,0);
    \begin{scope}[xshift=2.5cm]
    \draw[gray,dashed] (0,0) to (-1.333,0);
    \draw[gray,fill=gray!20] (0,0) circle (.333cm);
    \node[point,blue,fill=blue!20] (B) at (-.8,0) {};
    \draw[strand,blue] (B) to[out=270,in=180] (0,-.8) to[out=0,in=270] (.8,0) to[out=90,in=0] (0,.8) to[out=180,in=90] (B);
    \end{scope}
\end{tikzpicture}
\caption{The reductions for the proof of \cref{prop:equivalent:semi:reduced}.}
\label{fig:confluence}
\end{figure}
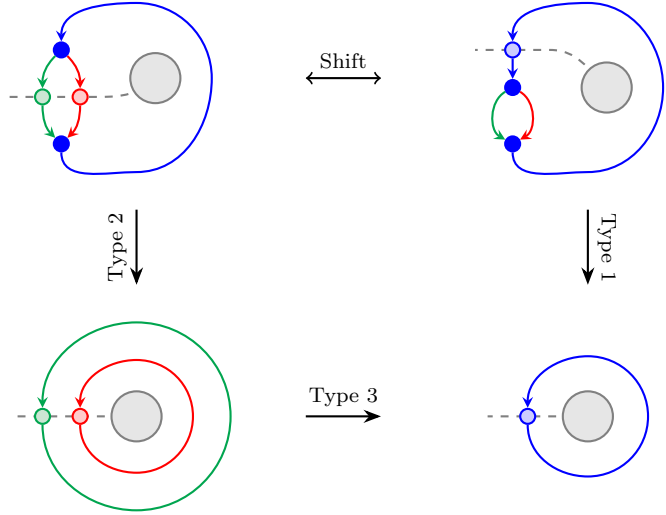

\subsubsection{A partial algorithm for the conjugacy problem}

Let us describe the algorithm.
We will later discuss its three steps.
Recall that $F$ and $G$ are conjugate if and only if $\C(F)$ and $\C(G)$ are equivalent (\cref{prop:conjugacy:is:equivalence}), so it suffices to decide the equivalence of $\Gamma$-closed strand diagrams.

\begin{algorithm}
\label{alg:main}
Let $\C(F)$ and $\C(G)$ be $\Gamma$-closed strand diagrams.
\begin{description}
    \item[Step 1:] Compute semi-reduced $\Gamma$-closed strand diagrams $\C(F^*)$ and $\C(G^*)$ that are equivalent to $\C(F)$ and $\C(G)$, respectively.
    \item[Step 2:] Check whether the split-merge parts of $\C(F^*)$ and $\C(G^*)$ are the same up to similarity.
    \item[Step 3:] Check whether the loop parts of $\C(F^*)$ and $\C(G^*)$ are the same up to similarities and type 3 reductions.
\end{description}
If the answers to both steps 2 and 3 are positive, then $\C(F)$ and $\C(G)$ are equivalent.
Otherwise, they are not by \cref{prop:equivalent:semi:reduced}.
\end{algorithm}

Let us comment on how each step works.

\textbf{Step 1:}
This is done by applying reductions of type 0, 1 and 2 (up to similarities when needed to ``unlock'' reductions).
Such procedure works because reductions of type 0, 1 and 2 (up to similarities) produce unique semi-reduced 3-equivalence classes, as was shown in \cref{prop:equivalent:semi:reduced}.

\textbf{Step 2:}
This is equivalent to checking whether two cocycles are cohomologous in the underlying graphs, as noted in Remark 3.1 in \cite{RearrConj}.
See also \cite{Nalecz}, which features explicit computations for Thompson's group $V$.

To discuss the third and last step of the algorithm, we need to introduce and discuss new tools.
We will do this right below, in a new section.


\section{Type 3 reductions and the loops semigroup}
\label{sec:loops}

Let us translate the problem of deciding whether two sets of loops are the same up to similarities and type 3 reductions into the word problem in a semigroup.

\subsection{Reductions of type 3 are not confluent}

Differently from how we tackled the split-merge part in \cref{prop:equivalent:semi:reduced}, understanding whether two loop parts are equivalent cannot be solved by just applying reductions until one finds a reduced object.
Indeed, type 3 reductions are generally not confluent, i.e., starting from the same object, one may make different choices of reductions and arrive at different reduced objects.

For example, let $\Gamma$ be the first graph depicted in \cref{fig:graph:non:confluent}.
Assume that a $\Gamma$-closed strand diagram has a loop part consisting of a red loop and a blue loop, both with winding number $1$.
There are two possible type 3 reductions:
one produces a single red loop and the other a single blue loop.
Thus, which reductions are performed matters for the purpose of comparing reduced objects.

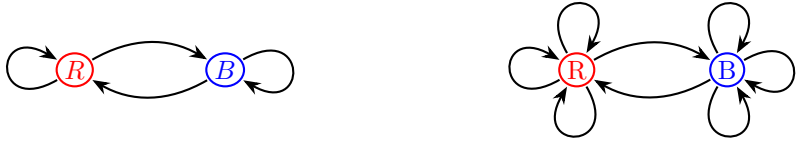
\begin{figure}
\centering
\begin{minipage}{.45\textwidth}
\centering
\begin{tikzpicture}
    \node[vertex,red] (R) at (-1,0) {$R$};
    \node[vertex,blue] (B) at (1,0) {$B$};
    \draw[edge] (R) to[out=30,in=150] (B);
    \draw[edge] (B) to[out=210,in=330] (R);
    \draw[edge] (R) to[out=210, in=150, min distance=1cm] (R);
    \draw[edge] (B) to[out=30, in=330, min distance=1cm] (B);
\end{tikzpicture}
\end{minipage}
\hfill
\begin{minipage}{.45\textwidth}
\centering
\begin{tikzpicture}
    \node[vertex,red] (R) at (-1,0) {R};
    \node[vertex,blue] (B) at (1,0) {B};
    \draw[edge] (R) to[out=30,in=150] (B);
    \draw[edge] (B) to[out=210,in=330] (R);
    \draw[edge] (R) to[out=300, in=240, min distance=1cm] (R);
    \draw[edge] (R) to[out=210, in=150, min distance=1cm] (R);
    \draw[edge] (R) to[out=120, in=60, min distance=1cm] (R);
    \draw[edge] (B) to[out=120, in=60, min distance=1cm] (B);
    \draw[edge] (B) to[out=30, in=330, min distance=1cm] (B);
    \draw[edge] (B) to[out=240, in=300, min distance=1cm] (B);
\end{tikzpicture}
\end{minipage}
\caption{Two graphs $\Gamma$ whose type 3 reductions are not confluent.}
\label{fig:graph:non:confluent}
\end{figure}

In the previous example, red and blue splits coincidentally have the same children, but there are non-confluent examples where this is not the case.
For instance, for the second graph $\Gamma$ depicted in \cref{fig:graph:non:confluent}, consider a loop part of $5$ red loops and $5$ blue loops.
We find three reduced objects:
(1) Two blue reductions followed by a red reduction result in a red loop.
(2) Two red reductions followed by a blue reduction result in a blue loop.
(3) A blue reduction followed by a red reduction (or the converse) result in two red loops and two blue loops.

Thus, we need other methods to decide when two loop parts are equivalent.

\subsection{The loops semigroup}

The strategy and techniques employed here are inspired by ideas that will appear in the upcoming work \cite{ConjNek}, which is currently work-in-progress.

Let us denote by $L(c,n)$ a loop colored by $c \in V\Gamma$ with winding number $n \geq 1$.
Given two multisets of loops, we compose them by uniting them.
Since they are multisets, we count with multiplicity:
a set of loops can have multiple copies of the same $L(c,n)$.
This operation is inherently associative and commutative.

For every $n \geq 1$ and every color $c \in V\Gamma$, let $\init^{-1}(c) = \{ e_1, \dots, e_k \}$ and $c_i = \term(e_i)$.
Reflecting type 3 reductions, we introduce the following relations:
\begin{equation}
\label{eq:relations}
\tag{$\mathsf{L}^{n}_{c}$}
L(c_1,n) + \dots + L(c_k,n) = L(c,n).
\end{equation}

The \textbf{loops semigroup} is the semigroup $\Loops$ generated by the infinite set
$$\left\{ L(c,n) \mid c \in V\Gamma, n \geq 1 \right\}$$
and subject to the relations described by \cref{eq:relations}.

The loop parts of $\Gamma$-closed strand diagrams are precisely sums of loops.
The fact that $\Loops$ is commutative and subject to the relations described by \cref{eq:relations} reflects how applying similarities and type 3 reductions to a sum of loops does not change the corresponding element of $\Loops$.
Thus, the elements of $\Loops$ correspond to the loop parts up to similarities and type 3 reductions.
Therefore, solving the word problem of $\Loops$ allows us to solve the third step of \cref{alg:main} and ultimately conclude this paper.

\begin{proposition}
\label{prop:loops:semigroup:union}
The loops semigroup $\Loops$ is the increasing union $\bigcup_{N\in\mathbb{N}} \Loops_N$ of finitely presented commutative semigroups
$$\Loops_N = \left\langle L(c,n),%
\begin{array}{ll}
     & \forall c \in V\Gamma,\\
     & \forall n=1,\dots,N
\end{array}
\ \bigg|\ \ \cref{eq:relations},
\begin{array}{ll}
     & \forall c \in V\Gamma\\
     & \forall n=1,\dots,N
\end{array}
\right\rangle.$$
\end{proposition}

\begin{proof}
Let $\Loops_N$ be the subset of $\Loops$ comprised of any sum of loops whose winding numbers are at most $N$.
The relations described by \cref{eq:relations} do not change the winding numbers, so each representative for an element of $\Loops_N$ is comprised of loops of winding number at most $N$.
Summing sets of loops with winding numbers at most $N$ produces a set of loops with winding numbers at most $N$, so $\Loops_N$ is a subsemigroup of $\Loops$.
The relations defining $\Loops_N$ are induced by those of $\Loops$, so $\Loops_N$ is presented as in the statement.
By definition $\Loops_N \subseteq \Loops_{N+1}$.
Finally, every element of $\Loops$ has a maximum winding number, so $\Loops = \bigcup_{N\in\mathbb{N}} \Loops_N$.
\end{proof}

The author could not find any notion of decomposition of commutative semigroups in the literature that reflects this description of $\Loops$.

Recall that finitely presented commutative semigroups have solvable word problem, see for example Theorem 5.1 of \cite{CommutativeSemigroup}.
Thus, the word problem of $\Loops$ itself is solvable using the following algorithm.

\begin{algorithm}
Let $L_1$ and $L_2$ be two elements of $\Loops$.
\begin{enumerate}
    \item Compute $N$ the maximum winding number among the loops of $L_1$ and $L_2$.
    Then $L_1$ and $L_2$ belong to the subsemigroup $\Loops_N$.
    \item Decide whether $L_1$ and $L_2$ are the same element in $\Loops_N$ using the fact that it is a finitely presented commutative semigroup.
\end{enumerate}
\end{algorithm}

This solves step 3 of \cref{alg:main}, concluding this paper.


\section*{Acknowledgements}

The author is thankful to Davide Perego and Francesco Matucci for fruitful conversations, to Julio Aroca, Jim Belk and Francesco Matucci for the ideas that arose from \cite{ConjNek} and to Collin Bleak, Noah Pugh and Teiva Jabbour for help finding literature about commutative semigroups and their word problem.

The author is supported by the "GoFR" project ANR-22-CE40-0004 and by the research grant PID2022-138719NA-I00 (Proyectos de Generación de Conocimiento 2022) financed by the Spanish Ministry of Science and Innovation, and is a member of the Gruppo Nazionale per le Strutture Algebriche, Geometriche e le loro Applicazioni (GNSAGA) of the Istituto Nazionale di Alta Matematica (INdAM).


\printbibliography[heading=bibintoc]

\end{document}